\documentclass[english]{amsart}
\usepackage[hmarginratio={1:1},vmarginratio={1:1},lmargin=60.0pt,tmargin=50.0pt]{geometry}
\usepackage[T1]{fontenc}
\usepackage[utf8]{inputenc}
\usepackage{csquotes}
\usepackage{float}
\usepackage{amsmath}
\usepackage{amsthm}
\usepackage{amssymb}
\usepackage{thmtools}
\usepackage{hyperref}
\usepackage{tikz}
\usetikzlibrary{arrows.meta}
\usepackage{xcolor}
\usepackage{biblatex}
\usepackage{ytableau}
\usepackage{subcaption}
\usepackage{graphicx}

\makeatletter
\theoremstyle{plain}
\newtheorem{thm}{\protect\theoremname}[section]
\newtheorem*{thm*}{\protect\theoremname}
\theoremstyle{definition}
\newtheorem{example}[thm]{\protect\examplename}
\theoremstyle{plain}
\newtheorem{prop}[thm]{\protect\propositionname}
\theoremstyle{plain}
\newtheorem{cor}[thm]{\protect\corname}
\theoremstyle{plain}
\newtheorem{lem}[thm]{\protect\lemmaname}
\theoremstyle{plain}

\theoremstyle{definition}
\newtheorem{problem}[thm]{\protect\problemname}
\theoremstyle{definition}
\newtheorem{optproblem}[thm]{\protect\optproblemname}
\theoremstyle{remark}
\newtheorem{rem}[thm]{\protect\remarkname}
\theoremstyle{definition}
\newtheorem{definition}[thm]{\protect\definitionname}

\makeatother

\usepackage{babel}
\providecommand{\conjecturename}{Conjecture}
\providecommand{\examplename}{Example}
\providecommand{\lemmaname}{Lemma}
\providecommand{\problemname}{Question}
\providecommand{\optproblemname}{Problem}
\providecommand{\propositionname}{Proposition}
\providecommand{\remarkname}{Remark}
\providecommand{\theoremname}{Theorem}
\providecommand{\corname}{Corollary}
\providecommand{\definitionname}{Definition}

\DeclareMathOperator{\im}{im}

\begin{document}
\title{Kazhdan--Lusztig basis and Optimization}

\author{Tom Goertzen and Geordie Williamson}
\address{T.G.: School of Mathematics and Statistics, University of Sydney, NSW 2006, Australia}
\email{tom.goertzen@sydney.edu.au}

\address{G.W.: School of Mathematics and Statistics, University of Sydney, NSW 2006, Australia}
\email{g.williamson@sydney.edu.au}
\date{}
\begin{abstract}
We describe a conjectural approach to obtaining canonical bases of the Hecke algebra at $q=1$ via continuous quadratic optimization. We focus on Specht modules $S^\lambda$ and proper cones inside $S^\lambda$ that are invariant under the action of $1+s$ for all simple reflections $s\in S$. We show that there are unique minimal and maximal cones invariant under all  $1+s$. For hook shapes, two-column shapes, and partitions of the form $(n-2,2)$, we prove that the Kazhdan--Lusztig basis spans this maximal cone. More generally, we define an optimization problem over bases that are unitriangular with respect to the polytabloid basis, subject to the constraint that the operators $1+s$ act non-negatively. We prove that the feasible region forms a compact semialgebraic set, and interpret it in terms of a hierarchy of invariant cones under all $1+s$. We demonstrate that minimizing the trace of the Gram matrix uniquely recovers Young's seminormal basis. Furthermore, we verify computationally that maximization uniquely recovers the Kazhdan--Lusztig basis for all partitions of $n\leq 7$. In higher ranks, the optimization detects deviations from the Kazhdan--Lusztig basis and may favour other natural positive bases, such as the Springer basis or $p$-canonical bases. Finally, we extend this framework to irreducible representations of $\mathfrak{sl}_n$. We observe that the Gelfand--Tsetlin basis corresponds to the unique minimizer, and we conjecture that the canonical basis corresponds to the maximum in small ranks.
\end{abstract}


\maketitle

\section{Introduction}

Canonical bases are central objects in modern representation theory. In most settings, they are obtained as the unique solutions of two constraints: ``self-duality'' and ``asymptotic orthogonality''. Typically, they also admit deep connections to representation theory and geometry. These connections explain their lasting appeal, and also yield deep properties of these bases, such as positivity and unimodality.

It is an interesting and important problem to ask whether there exist other ways of defining, computing, or understanding canonical bases. A prominent example of the fruitfulness of this question is the theory of cluster algebras, which arose from Fomin and Zelevinsky's attempts to understand multiplicative properties of the dual canonical basis of quantum groups \cite{ClusterAlgebras}. Another example, which was a central motivation for this work, is provided by Stembridge's conjectures on admissible $W$-graphs \cite{Stembridge2008}. These conjectures yield a simple combinatorial way of recovering the $W$-graphs of left-cell representations of symmetric groups, without computing Kazhdan--Lusztig polynomials. Stembridge's conjectures were recently resolved in type $A$ in important work of Nguyen \cite{Nguyen2020}, who proved that the set of admissible cells coincides with the set of Kazhdan--Lusztig left cells.

The main point of this paper is to advertise the fact that canonical bases are interesting from the point of view of optimization. We concentrate on the Kazhdan--Lusztig basis, introduced in \cite{KazhdanLusztig79}, and find various settings in which this basis either provably or experimentally solves an optimization problem. An interesting feature of our approach is that we often incorporate deep properties of the basis into the \emph{constraints} of the optimization problem. For example, we prove in several settings that the Kazhdan--Lusztig basis is the ``most efficient'' basis which is positive in the standard basis and has positive structure constants.

We are not able to prove much in great generality, but we do think the observation is an important one. The paper contains many open problems which might be tractable. We suspect that the interested reader could go much further with the right idea.

\subsection{Main results}

Our primary focus is in type $A$, specifically the relationship between the Kazhdan--Lusztig left cell representation and Young's seminormal representation within the Specht module $S^{\lambda}$. Both bases share two critical features:
\begin{enumerate}
    \item The transition matrix from the standard polytabloid basis is upper unitriangular, with respect to the last letter order.
    \item The action of the operators $1+s$ for $s \in S$ is non-negative in these bases.
\end{enumerate}

We formulate a constrained optimization problem on the set of such bases. Let $A$ be an upper unitriangular matrix representing a change of basis from the polytabloid basis. We define the objective function
\[
f(A) = \mathrm{Tr}(A^t G A),
\]
where $G$ is the Gram matrix of the invariant form on $S^{\lambda}$. We optimize $f(A)$ subject to the constraint $A^{-1}(1+s)A \ge 0$. Our results establish a duality between these two canonical bases.

\subsubsection*{Minimization (Young's Seminormal Form)}
We prove that Young's seminormal basis is the unique minimizer of this optimization problem.

\begin{thm*}
The function $f(A)$ has a unique minimum under the non-negativity constraints. This minimum is achieved by Young's seminormal basis.
\end{thm*}

We remark that this result is not deep: it is a consequence of the fact that diagonalizing a matrix can be expressed as a minimization problem. Since Young's seminormal basis is known to diagonalize the Gram matrix (being orthogonal with respect to the invariant form \cite{Murphy81}), it naturally appears as the solution to this trace minimization.

\subsubsection*{Maximization (Kazhdan--Lusztig Basis)}
We conjecture that the Kazhdan--Lusztig left cell basis corresponds to the maximum in low ranks. In higher ranks, the optimizer detects geometric anomalies. While a general statement and proof remains open, we provide both analytical and numerical evidence.

\begin{thm*}
For the reflection representation $\lambda=(n-1,1)$, the base-change matrix to the Kazhdan--Lusztig left cell representation is a critical point of the optimization problem.
\end{thm*}

We have further verified computationally that the Kazhdan--Lusztig basis corresponds to the unique global maximum for all partitions of $n \le 7$. In higher ranks, however, Specht modules admit other natural positive bases that generally differ from the Kazhdan--Lusztig basis and can yield strictly larger values in the optimization problem. For example, both the Springer basis and $p$-canonical bases typically differ from the Kazhdan--Lusztig basis and give rise to larger cones and larger objective values.

\subsubsection*{Geometric Interpretation (Invariant Cones)}
We interpret these results via the theory of invariant cones. The non-negative action of $1+s$ guarantees the existence of invariant cones. Irreducibility implies that there exist unique minimal and maximal invariant cones \cite{MejstrikProtasov2025}.

\begin{thm*}
For hook-shape, two-column and $(n-2,2)$ partitions, the Kazhdan--Lusztig basis generates the unique maximal invariant cone inside $S^{\lambda}$.
\end{thm*}

A striking feature of our results is that for all partitions of $n \le 8$, the maximal invariant cone is simplicial and is generated by the Kazhdan--Lusztig basis. We show that this beautiful geometric property persists exactly until the partition $\lambda=(4,4,1)$. In this case, we verify that the maximal invariant cone is no longer simplicial. This reveals that while the quadratic optimization approach works robustly for small rank, the underlying geometric structure of the invariant cones becomes significantly more complex beyond this point. 

\subsubsection{Generalization to Canonical Bases}

In Section \ref{sec:canonical}, we apply this optimization perspective to finite-dimensional representations of the Lie algebra $\mathfrak{sl}_n$. We define an analogous problem over bases that are unitriangular with respect to the Gelfand--Tsetlin basis. We show that the Gelfand--Tsetlin basis is the unique minimizer of the trace of the Gram matrix. Furthermore, we conjecture that the canonical basis corresponds to the unique maximum in small ranks $n\leq 4$. We have verified this conjecture for all $\mathfrak{sl}_4$ modules $V(\lambda)$ with $\lambda \vdash m$ and $m \leq 10$. In general, it seems fruitful to compare different bases giving rise to invariant cones.

\subsection{Future directions}

To conclude the introduction, let us discuss several other motivations and avenues for future work.

\subsubsection{$p$-canonical bases.}
A major motivation for the current work is the theory of $p$-canonical bases in modular representation theory \cite{JensenWilliamson2017}. These are a central repository for important invariants of modular representations (dimensions of simples, multiplicities, character values, etc.). Current algorithms exist but are not ``combinatorial'' in the same way that algorithms for computing traditional canonical bases are. Whilst the Kazhdan--Lusztig basis can be computed via a simple algorithm in the Hecke algebra, algorithms to compute the $p$-canonical basis make use of difficult computations in the diagrammatic Hecke category, and are orders of magnitude slower \cite{GibsonJensenWilliamson2023}.

A basic property of the $p$-canonical basis is that it is positive in the Kazhdan--Lusztig basis. Experiments of the second author showed that in many settings where the $p$-canonical basis is known, it can be computed as the solution of a linear program, with constraints arising from elementary positivity properties of the $p$-canonical basis. The linear programming approach is orders of magnitude faster than the original algorithm, but we have no guarantee that it yields the right answer. Thus, it is an important question, closely tied to the results of this paper, when and why the approach via optimization works.

\subsubsection{Optimization.}
The search for canonical bases typically leads to very large optimization problems. For example, the task of finding structure constants leads to an optimization problem where there is one coefficient for each pair of permutations of $n$, leading to a naive parameter count of $(n!)^2$. Whilst linear solvers can handle such high-dimensional problems with millions of variables, they struggle with the quadratic optimization problems that arise here. It would be very interesting to attempt gradient-based methods \cite{BazaraaSheraliShetty2006} or ADMM \cite{ADMM}. Another promising approach is to use semidefinite programming (SDP)  to obtain relaxations of the quadratic optimization problem, see \cite{Schweighofer2005}. We hope to return to this in future work.

\subsubsection{Reinforcement learning.}
Reinforcement learning is the field of AI used to solve problems like Go and Chess \cite{RL}. In this setting, one has a set of possible moves (in a known or unknown environment) and the goal is to maximize the possible reward. Many algebraic algorithms to compute canonical bases resemble a game. For example, in the setting of the Hecke algebra the key algebraic step for computing Kazhdan--Lusztig bases involves the formula
\[
b_sb_x = b_{sx} + \sum_{y < x; sy < y} \mu(y,x)b_y.
\]
Here, it is tempting to regard the left-hand side as given, and the right-hand side as a choice of decomposition. One would like to decompose $b_sb_x$ into basis elements, which leads to positive future rewards, perhaps expressed via small but positive $\mu(y,x)$. In future work, we hope to run reinforcement learning algorithms on Hecke algebras and quantum groups, in the hope of recovering canonical bases. From this point of view, the current work serves as a baseline to see how far one can get with elementary approaches.

\section{Preliminaries}

\subsection{Coxeter systems}

A Coxeter system $(W,S)$ consists of a set of simple reflections $S$ and a group $W$, together with positive integers $m_{s,t}\in\mathbb{Z}_{\geq2}\cup\{\infty\}$ for $s\neq t$ and $m_{ss}=1$ such that $W$ has the following presentation
\[
W=\langle s\in S\mid(st)^{m_{st}}=1\forall s,t,\in S\rangle.
\]
We denote by $T$ the set of reflections, i.e. $T=\{w^{-1}sw\mid s\in S,w\in W\}$. Let $\ell:W\to\mathbb{Z}_{\geq0}$ be the length function of $W$, i.e. for $w\in W$ we have $$\ell(w)=\min\{n\mid\exists s_{1},\dots,s_{n}\in S:s_{1}\cdots s_{n}=w\}.$$ The Bruhat order, denoted by $x\leq y$  for $x,y\in W $, is the transitive closure of the relations $x\to y$ if $\ell(x)\leq\ell(y)$ and there exists $t\in T$ such that $xt=y$. The Bruhat graph is the graph with vertices $W$ and edges given by the relation $\to$. For finite Coxeter groups $W$ there exists a unique element of maximal length (the longest element) $w_{0}\in W$, which is also maximal in the Bruhat order. For a detailed treatment of Coxeter systems, we refer to \cite{Humphreys90,BjornerBrenti2005}.
\begin{example}
Symmetric groups are Coxeter groups. A Coxeter system for $S_{n}$ is given by $(W,S)=(S_{n},\{(i,i+1)\mid i=1,\dots,n-1\})$ and $T=\{(i,j)\mid1\leq i<j\leq n\}$. The Bruhat graph of $S_3$ is given in Figure \ref{fig:bruhat_graph_a2}.
\begin{figure}[H]
    \centering
    \begin{tikzpicture}[
    vertex/.style={inner sep=2pt, font=\normalsize, draw},
    edge/.style={-Stealth, thick}
]

\node[vertex] (id) at (0,0) {$id$};
\node[vertex] (s) at (-2,1.732) {$s$}; 
\node[vertex] (t) at (2,1.732) {$t$};
\node[vertex] (st) at (2,3.464) {$st$}; 
\node[vertex] (ts) at (-2,3.464) {$ts$};
\node[vertex] (sts) at (0,5.196) {$sts=tst$}; 

\draw[edge] (id) -- (s);
\draw[edge] (id) -- (t);
\draw[edge] (id) -- (sts);
\draw[edge] (s) -- (st);
\draw[edge] (t) -- (st);
\draw[edge] (t) -- (ts);
\draw[edge] (s) -- (ts);
\draw[edge] (st) -- (sts);
\draw[edge] (ts) -- (sts);

\end{tikzpicture}
    \caption{Bruhat graph in type $A_2$.}
    \label{fig:bruhat_graph_a2}
\end{figure}
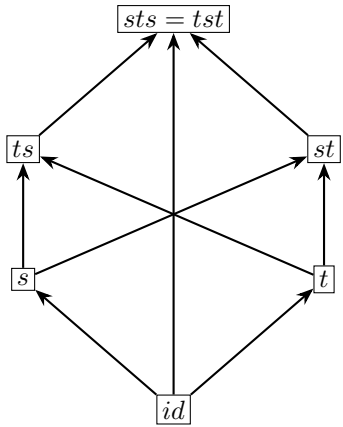
\end{example}

\subsection{Hecke algebra and Kazhdan--Lusztig basis}

The Hecke algebra associated to the Coxeter system $(W,S)$ is a unital associative algebra over $\mathbb{Z}[v,v^{-1}]$, denoted by $\mathcal{H}=\mathcal{H}(W)$. It has generators $\delta_{s},s\in S$ subject to the following relations:
\begin{enumerate}
\item $\delta_{s}^{2}=(v^{-1}-v)\delta_{s}+1$ (quadratic relations);
\item $\underbrace{\delta_{s}\delta_{t}\cdots}_{m_{st}{\rm -times}}=\underbrace{\delta_{t}\delta_{s}\cdots}_{m_{st}{\rm -times}}$ (braid relations).
\end{enumerate}
Using reduced expressions in $W$, we can give a basis for $\mathcal{H}$, as described in the following proposition.
\begin{prop}
The elements $\delta_{x}=\delta_{s_{1}}\cdots\delta_{s_{n}}$ for $x\in W$ and any reduced expression $x=s_{1}\cdots s_{n}$ with $n=\ell(x)$
give a $\mathbb{Z}[v,v^{-1}]$-basis of $\mathcal{H}$, called the standard basis.
\end{prop}
\begin{rem}
    In the literature, the $\delta_{s}$ are sometimes denoted by $H_{s}$  \cite{SoergelCombinatoricsTilting97} or $\tilde{T}_s$ \cite{Lusztig83} instead.  It is also common to consider generators $\delta_{s}=vT_{s},$ and $q=v^{-2}$ at the cost of working over $\mathbb{Z}[q^{1/2},q^{-1/2}]$, e.g. \cite{KazhdanLusztig79}.
\end{rem}

In  \cite{KazhdanLusztig79}, Kazhdan and Lusztig define an involution  on $\mathcal{H}$, called \emph{bar involution}, by setting $\overline{v}=v^{-1},\overline{\delta}_{x}=\delta_{x^{-1}}^{-1}$. For example, we have $\overline{\delta_{s}}=\delta_{s}+(v-v^{-1})$ and $\overline{\delta_{s}+v}=\delta_{s}+v$.
\begin{thm}[\cite{KazhdanLusztig79,SoergelCombinatoricsTilting97}]
There exists a unique bar invariant basis $\{b_{x}\mid x\in W\}$ of $\mathcal{H}$ such that 
\[
b_{x}=\delta_{x}+\sum_{y<x}h_{y,x}\delta_{y},
\]
where $h_{y,x}\in v\mathbb{Z}[v]$. 
\end{thm}
This basis is called the Kazhdan--Lusztig basis of $\mathcal{H}$ and $h_{y,x}$ is called the Kazhdan--Lusztig polynomial associated to $x,y\in W$. We also set $h_{x,x}=1$ and $h_{y,x}=0$ if $y\not\leq x$.
\begin{rem}
    In the original notation of \cite{KazhdanLusztig79}, we have $b_{x}=C_{x}',$ and $h_{y,x}=v^{\ell(x)-\ell(y)}P_{y,x}$.
\end{rem}

There are at least two prominent ways of obtaining the Kazhdan--Lusztig basis. The original method due to Kazhdan and Lusztig involves the so-called $R$-polynomials. Another way (see \cite{SoergelCombinatoricsTilting97}), which does not involve the $R$-polynomials, involves the recursion
\[
b_{xs}=b_{x}b_{s}-\sum_{y<x}h'_{y}(0)b_{y},
\]
where the polynomials $h'_{y}$ satisfy
\[
b_{x}b_{s}=b_{xs}+\sum_{y<x}h'_{y}\delta_{y}.
\]
\begin{example}
    For $s\in S$, we have $b_{s}=\delta_{s}+v$. When $W$ is finite, we have  $b_{w_{0}}={\displaystyle \sum_{y\in W}v^{\ell(w_{0})-\ell(y)}\delta_{y}}$, see for instance  \cite{SoergelCombinatoricsTilting97}. More generally, for the longest element of any parabolic subgroup $W_f$ with longest element $w_f$, we have $b_{w_f}={\displaystyle \sum_{y\in W_f}v^{\ell(w_{f})-\ell(y)}\delta_{y}}$.
\end{example}
We denote the structure constants of the Kazhdan--Lusztig basis by  $\mu_{xy}^z$, i.e. $b_xb_y=\displaystyle\sum_{z\in W}\mu_{xy}^zb_z$.
The Kazhdan--Lusztig positivity conjecture states that the coefficients of the Kazhdan--Lusztig polynomials are non-negative, i.e.\  $h_{y,x}\in v\mathbb{Z}_{\geq0}[v]$ and moreover that the structure constants are of the form $\mu_{xy}^z\in\mathbb{Z}_{\geq0}[v]$. This was proven by Kazhdan and Lusztig in \cite{KazhdanLusztig80} for finite and affine Weyl groups and in full generality for arbitrary Coxeter groups by Elias and the second author in \cite{EliasWilliamson14}.

We can interpret $h_{y,x}(1)$ as multiplicities of simple modules within Verma modules, via the Kazhdan--Lusztig conjecture \cite{KazhdanLusztig79}, which became a theorem due to Beilinson and Bernstein \cite{BeilinsonBernstein1981} and independently by Brylinski and Kashiwara \cite{BrylinskiKashiwara} using the theory of $\mathcal{D}$-modules.

Both recursions for computing Kazhdan--Lusztig polynomials involve either evaluating polynomials or using the bar involution. Note that these operations become ``invisible'' when setting $v=1$. Thus, a priori, the Kazhdan--Lusztig basis is ``invisible'' in the group ring $\mathbb{Z}[W]$. This leads naturally to the following question, which has been around ever since the discovery of the Kazhdan--Lusztig basis:
\begin{problem}\label{prob:intrinsic_KL}
Can we compute $h_{y,x}(1)$ inside the group ring of $W$ intrinsically (without going through the Hecke algebra $\mathcal{H}$)?
\end{problem}

A positive answer to Question \ref{prob:intrinsic_KL} would be very interesting. It would necessarily entail a completely different way of defining or computing the Kazhdan--Lusztig basis.

\subsection{W-graphs and Kazhdan--Lusztig left cell representations}

An important application of the Kazhdan--Lusztig basis is to obtain representations of the Hecke algebra and the underlying Coxeter group. This was achieved in the original paper of Kazhdan and Lusztig \cite{KazhdanLusztig79}, where it is explained how the action may be encoded in terms of certain sparse graphs called $W$-graphs. We use the conventional $W$-graph notation of \cite{KazhdanLusztig79}, in which the Hecke algebra generators are written with respect to an alternative Kazhdan--Lusztig basis $\{\tilde b_x\}$ defined below. The action on the previously defined basis $\{b_x\}$ is given by the transpose matrices, i.e.\ the dual representation. In finite type, this follows from \cite[Corollary~3.2]{KazhdanLusztig79}.

\begin{definition}[\cite{KazhdanLusztig79,Lusztig83}]
    A $W$-graph $\Gamma=\Gamma(X,Y,I,\mu)$ is a graph with vertices $X$, edges $Y$ and additional data
    \begin{itemize}
        \item labelling of the vertices $I(x)\subseteq S$, $x\in X$,
        \item labelling of the edges by non-zero integers $\mu(x,y),x,y\in X$,
    \end{itemize} such that the free $\mathbb{Z}[v,v^{-1}]$-module with basis indexed by $X$ yields a representation of $\mathcal{H}(W)$ via the following action
$$\delta_{s}{x}=\begin{cases}
-v{x}, &s\in I(x),\\
v^{-1}x+\sum_{\{y\in Y \mid s\in I(y)\}}\mu(y,x)y, & s\notin I(x).
\end{cases}$$
\end{definition}

\begin{example}
    The Kazhdan--Lusztig basis yields a $W$-graph of the regular representation by choosing $X=W$, and edges of the form $\{x,y\},$ whenever $b_y$ appears in $b_sb_y$ for some $s\in S$, see \cite{KazhdanLusztig79}. The vertex labels are given by the left descent set $\mathcal{L}(x)=\{s\in S\mid sx <x\}$. The edge labels  $\mu(y,x)$ are given as the coefficient of $v$ in $h_{y,x}$. This $W$-graph structure comes from the $\mathcal{H}$-action on another variant of the Kazhdan--Lusztig basis, given via $$\tilde{b}_x=\sum_{y\leq x}(-1)^{\ell(x)+\ell(y)}\overline{h_{x,y}}\delta_y,$$see \cite{KazhdanLusztig79,SoergelCombinatoricsTilting97}. Note that in the language of \cite{KazhdanLusztig79}, we have $C_x'=b_x,C_x=\tilde{b}_x$, with $v^{-2}=q$ and Hecke algebra generators $v^{-1}\delta_s=T_s$.
\end{example}

We can define the following preorder as the transitive closure of the following relations (see \cite{Lusztig83}):
$$y\leq_Lx \text{ if }b_y\text{ appears in }b_sb_x\text{with non-zero coefficient for some }s\in S.$$
We write $x\sim_Ly$ for the corresponding equivalence relation, i.e.\ $x\sim_Ly$ if $x\leq_Ly$ and $y\leq_Lx$. The equivalence classes under this relation are called left cells. In \cite{KazhdanLusztig79} it is proved that each left cell defines a representation of the Hecke algebra, called the Kazhdan--Lusztig left cell representation.  Note that $$\mathcal{H}_{\leq_L x}=\langle \tilde{b}_y \mid y \leq_Lx\rangle, \mathcal{H}_{<_L x}=\langle \tilde{b}_y \mid y <_Lx\rangle $$are left ideals in the Hecke algebra $\mathcal{H}$. The left cell representation is defined as the quotient $\mathcal{H}_{\leq_L x}/\mathcal{H}_{<_L x}$ and can be encoded using a $W$-graph, see \cite{KazhdanLusztig79}. When we set $v=1$, we obtain a representation for the underlying Coxeter group. In type $A$, the combinatorics of Kazhdan--Lusztig left cells are closely tied to the combinatorics of Young tableaux.

\subsection{Young Tableaux}

Let $n$ be a positive natural number and $\lambda$ be a partition of $n$, i.e., $\lambda=(\lambda_1,\dots,\lambda_k)$ with $\lambda_1\geq \dots \geq \lambda_k>0$ and $\sum_{i=1}^k\lambda_i=n$, which we denote by $\lambda \vdash n$.

We associate to $\lambda$ a \emph{Young diagram}, which consists of a collection of boxes arranged in left-justified rows, with $\lambda_i$ boxes in the $i$-th row. The \emph{conjugate partition} of $\lambda$, denoted by $\lambda'$ is obtained by swapping the rows and columns of its Young diagram. A \emph{Young tableau} of shape $\lambda$ is a filling of the boxes of the Young diagram with the integers $\{1, \dots, n\}$ such that each number appears exactly once. We denote the set of all Young tableaux of shape $\lambda$ by $\mathrm{YT}(\lambda)$. The symmetric group $S_n$ naturally acts on the set of Young tableaux on the left by applying an element $\sigma\in S_n$ to each entry simultaneously.

A \emph{standard Young tableau} is a Young tableau in which the entries are strictly increasing along each row and down each column. We denote the set of all standard Young tableaux of shape $\lambda$ by $\mathrm{Std}(\lambda)$.

For a standard Young tableau $T \in \mathrm{Std}(\lambda)$, a number $i \in \{1, \dots, n-1\}$ is called a \emph{descent} of $T$ if $i+1$ lies strictly south of $i$ (i.e., in a row with a larger index). In this context, strictly south implies weakly left. We denote the set of descents by:
\[
    D(T) = \{ i \mid i+1 \text{ lies strictly south of } i \}.
\]
We define a total order on $\mathrm{Std}(\lambda)$ called the \emph{last letter order}, denoted by $<$. Let $T, T' \in \operatorname{Std}(\lambda)$. We say $T < T'$ if the entry $n$ lies in a row with a strictly larger index in $T'$ than in $T$. If $n$ is in the same row in both tableaux, we ignore $n$ and compare the positions of $n-1$ recursively. This ordering is essential for the representation theory of the symmetric groups.

\subsection{Kazhdan--Lusztig left cell representations in type \texorpdfstring{$A$}{A}}

In type $A$, it was shown in \cite{KazhdanLusztig79} that the isomorphism types of left cells correspond exactly to the irreducible representations of $W\cong S_n$, which are labelled by partitions of $n$. Thus, each irreducible module has a representation coming from a left cell and for a given partition $\lambda\vdash n$, we denote the $W$-graph corresponding to such a left cell by $\Gamma^\lambda$. For example, in Figure \ref{fig:3_2} we show the $W$-graph $\Gamma^{(3,2)}$.
\begin{figure}
    \centering
    \resizebox{.3\linewidth}{!}{\ytableausetup{centertableaux, notabloids, smalltableaux}

\begin{tikzpicture}[
    vertex/.style={rectangle, draw=none, inner sep=2pt},
    edge/.style={thick}
]


\node[vertex] (2) at (2, 2) {
    \begin{ytableau}
    1 & 3 & 4 \\
    2 & 5
    \end{ytableau}
};

\node[vertex] (0) at  (0, 4) {
    \begin{ytableau}
    1 & 3 & 5 \\
    2 & 4
    \end{ytableau}
};

\node[vertex] (3) at (0, 0) {
    \begin{ytableau}
    1 & 2 & 4 \\
    3 & 5
    \end{ytableau}
};

\node[vertex] (1) at (-2, 2) {
    \begin{ytableau}
    1 & 2 & 5 \\
    3 & 4
    \end{ytableau}
};

\node[vertex] (4) at (0, -2) {
    \begin{ytableau}
    1 & 2 & 3 \\
    4 & 5
    \end{ytableau}
};


\draw[edge] (1) -- (0); 
\draw[edge] (0) -- (2); 
\draw[edge] (2) -- (3); 
\draw[edge] (1) -- (3); 

\draw[edge] (3) -- (4); 

\draw[edge] (0) to[bend left=90, looseness=1.5] (4); 

\end{tikzpicture}}
    \caption{$W$-graph $\Gamma^{(3,2)}$: All edge weights are equal to $1$. For the correspondence to the vertex labelling via descents of standard tableaux, see Table \ref{tab:correspondence}.}
    \label{fig:3_2}
\end{figure}
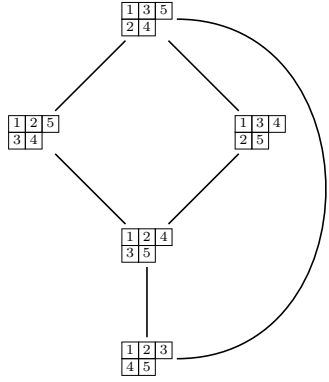
If another representation yields the same matrix representation as a Kazhdan--Lusztig left cell representation, we still refer to the module equipped with its underlying basis as a Kazhdan--Lusztig left cell representation.

\begin{rem}\label{rem:descents}
    We can use the Robinson-Schensted correspondence to decide whether two elements $x,y\in W$ lie in the same left cell. The Robinson-Schensted correspondence associates to $w\in S_n$ a unique pair of standard Young tableaux of the same shape with $n$ boxes, i.e. we have a bijective map    $$\mathrm{RS}:S_n\to\bigcup_\lambda \mathrm{Std}(\lambda),w\mapsto(P(w),Q(w)).$$ It is known that $x\sim_Ly$ if and only if $Q(x)=Q(y)$, see \cite{KazhdanLusztig79} and \cite{BjornerBrenti2005,GarsiaMcLarnan88} for a more detailed proof. Hence, the canonical basis elements of each left cell are indexed by standard Young tableaux $T\in \mathrm{Std}(\lambda)$ of the same shape. We can further deduce the vertex labelling of the associated $W$-graph using the descent set of the corresponding standard Young tableau.
\end{rem}

\begin{example}\label{ex:simple_roots}
    The reflection representation for the symmetric group $W=S_{n+1}$ corresponds to the partition $\lambda=(n,1)$. Each standard Young tableau of this shape has exactly one descent. We can obtain the Kazhdan--Lusztig left cell representation for $\lambda$  by considering the action in the simple root basis $\alpha_1,\dots,\alpha_n\in \mathbb{R}^n$ via $$s_i(\alpha_j)=\alpha_j-\langle\alpha_i,\alpha_j\rangle\alpha_i,$$where $\langle\alpha_i,\alpha_j\rangle$ is the $(i,j)$-entry of the Cartan matrix of type $A_n$. 
\end{example}

Other explicit descriptions for left cell representations and their corresponding invariant forms are given for hook shapes \cite{BjornerBrenti2005,Fung2003} or two-row shapes \cite{LascouxSchuetzenberger,Westbury95}.

In Sections \ref{sec:compute_cones} and \ref{sec:compute_left_cells}, we propose several ways of characterising Kazhdan--Lusztig left cell representations in type $A$ by setting $v=1$ and exploiting the fact that $1+s$ acts non-negatively on the Kazhdan--Lusztig left cell representation.

\subsection{Specht modules and the polytabloid basis}

Specht modules provide a concrete construction of all irreducible representations of $S_n$. Let $\lambda$ be a partition of $n$. For a Young tableau $T$ of shape $\lambda$, we denote its row and column stabilizers by $R(T)$ and $C(T)$, respectively. The \emph{Young symmetrizer} is defined as
$$b_T=\sum_{\sigma\in C(T)}\mathrm{sgn}(\sigma)\sigma.$$
Two tableaux $T$ and $T'$ are equivalent if $T' = \sigma T$ for some $\sigma \in R(T)$. An equivalence class is called a \emph{tabloid} and is denoted by $\{T\}$. The symmetric group $S_n$ acts on the set of tabloids via $\sigma \{T\}=\{\sigma T\}$.

Let $M^\lambda$ be the complex vector space spanned by the set of all tabloids of shape $\lambda$. We equip $M^\lambda$ with an $S_n$-invariant inner product by declaring the tabloids to be orthonormal, i.e., $(\{T\},\{T'\})=\delta_{\{T\},\{T'\}}$.
We define the \emph{polytabloid} associated to $T$ as $v_T=b_T\{T\}$. It satisfies $\sigma v_T=v_{\sigma T}$ for all $\sigma\in S_n$. The \emph{Specht module} $S^\lambda$ is the submodule of $M^\lambda$ spanned by the polytabloids:
$$S^\lambda=\langle v_T \mid T\in \mathrm{YT}(\lambda)\rangle_\mathbb{C}.$$
We list the fundamental properties of Specht modules below; for details see \cite[Chapter~7]{JamesKerber1981}.

\begin{thm}\label{thm:polytabloid}
    The modules $S^\lambda$ satisfy the following properties:
    \begin{enumerate}
        \item the set $\{S^\lambda \mid \lambda \vdash n\}$ provides a complete set of isomorphism classes of irreducible $S_n$-modules;
        \item the set of polytabloids labelled by standard Young tableaux, $\{v_T \mid T\in \mathrm{Std}(\lambda)\}$, forms a basis for $S^\lambda$;
        \item the matrices representing the action of $S_n$ in this basis have integer coefficients;
        \item the restriction of the inner product on $M^\lambda$ to $S^\lambda$ is positive definite and $S_n$-invariant.
    \end{enumerate}
\end{thm}
\subsection{Young's seminormal form}
Young's seminormal form provides a sparse, rational matrix representation for the irreducible representations of the symmetric group. We summarize the construction for simple reflections below, following \cite[Ch.~3]{JamesKerber1981}.

Fix a partition $\lambda\vdash n$. We order the set of standard tableaux $\mathrm{Std}(\lambda)$ according to the last letter order. For a tableau $T\in \mathrm{Std}(\lambda)$ and $i\in \{1,\dots,n-1\}$, we define the \emph{axial distance} $a_i(T)$ as the signed distance between $i$ and $i+1$ in $T$. Explicitly, if $c_k$ and $r_k$ denote the column and row indices of the entry $k$, then $a_i(T) = (c_{i+1} - r_{i+1}) - (c_i - r_i)$. For each $T \in \mathrm{Std}(\lambda)$, let $e_T$ be the corresponding basis vector.

\begin{thm}[Young's seminormal form]
    The vector space spanned by $\{e_T \mid T \in \mathrm{Std}(\lambda)\}$ affords an irreducible representation isomorphic to $S^\lambda$. The action of the simple reflection $s_i$ is given by:
    $$s_{i}e_{T}=\begin{cases}
    e_{T}, & i,i+1\text{ are in the same row},\\
    -e_{T}, & i,i+1\text{ are in the same column},\\
    a_{i}(T)^{-1}e_{T}+e_{s_{i}T}, & s_{i}T\in\mathrm{Std}(\lambda)\text{ and }T<s_{i}T,\\
    (1-a_{i}(T)^{-2})e_{s_{i}T}-a_{i}(T)^{-1}e_{T}, & s_{i}T\in\mathrm{Std}(\lambda)\text{ and }s_{i}T<T.
    \end{cases}$$
\end{thm}

\subsection{Dual representation and invariant forms}

Throughout the paper, we rely heavily on the following two basic constructions. Let $G$ be a finite group and $V$ be an $\mathbb{R}G$-module. We have the \emph{dual representation} of $V$, denoted $V^*=\mathrm{Hom}(V,\mathbb{R})$, with the linear $G$-action given by
\[
    G \times V^* \to V^*, \quad (g,f) \mapsto (v \mapsto f(g^{-1}v)).
\]
If $\rho \colon G \to \mathrm{GL}_d(\mathbb{R})$ is the matrix representation of $V$ for a fixed basis, the matrix representation for $V^*$ with respect to the dual basis is given by $g \mapsto \rho(g^{-1})^t$.

\begin{example}
    \begin{enumerate}
        \item Let $(W,S)$ be a Coxeter group and $\rho \colon W \to \mathrm{GL}_d(V)$ a matrix representation. Then the dual representation is determined by the images of the simple reflections. Since $s^2=1$ implies $s=s^{-1}$ for any $s \in S$, the dual action is simply given by $\rho^*(s) = \rho(s)^t$.
        
        \item Let $\Gamma=\Gamma(X,Y,I,\mu)$ be a $W$-graph and consider the corresponding $W$-representation at $v=1$. The dual representation acts on the dual basis $\{e_x\}_{x \in X}$ as follows:
        \[
        s_{i} \cdot e_x = \begin{cases}
             e_x, & i \in I^{c}(x), \\
            -e_x + \sum_{y: \, i \in I^c(y), \, x \sim y} \mu(x,y)e_{y}, & i \notin I^{c}(x),
        \end{cases}
        \]
        where $I^c(x)=S \setminus I(x)$ is the complement of the descent set at vertex $x$.\footnote{Note that in this normalization, the operator $1+s_i$ acts as multiplication by $2$ on the diagonal when $i \in I^c(x)$.}
    \end{enumerate}
\end{example}

Schur's Lemma and the averaging method for finite groups imply the following fact (see \cite[Sec. 13.1-13.2]{SerreLinRepsFiniteGroups} for the classification of real irreducible modules and their invariant forms).

\begin{lem}
    Let $G$ be a finite group and $V$ an irreducible $\mathbb{R}G$-module. Then there exists a positive definite symmetric bilinear $G$-invariant form on $V$, which is unique up to a positive scalar. Consequently, the form induces an isomorphism $V \cong V^*$ of $\mathbb{R}G$-modules.
\end{lem}

An example of such an invariant form is given for the Specht modules $S^\lambda$ in Theorem \ref{thm:polytabloid}.

\section{Experiments for Computing the Kazhdan--Lusztig Basis at \texorpdfstring{$v=1$}{v=1}}

Let $(W,S)$ be a Coxeter system with $W$ finite. Let $\dot{b}_{x}$ denote the specialization of the Kazhdan--Lusztig basis element $b_x$ at $v=1$. We summarize the following key properties of this specialization:
\begin{enumerate}
    \item $h_{y,x}(1) \in \mathbb{Z}_{\geq 0}$;
    \item $\dot{b}_{s} = s+1$;
    \item $\dot{b}_{x} \dot{b}_{y} = \sum_{z\in W} \mu_{xy}^{z} \dot{b}_{z}$, with $\mu_{xy}^{z} \in \mathbb{Z}_{\geq 0}$;
    \item $\dot{b}_{w_0} = \sum_{y\in W} y$.
\end{enumerate}
This leads to the following optimization problem.

\begin{optproblem}\label{conj:KLBasisAt1}
    For $x\in W,$ consider $A_{x}={\displaystyle \sum_{y\leq x}}a_{y,x}y\in\mathbb{R}[W]$ and minimize  $${\displaystyle \sum_{w\in W}\sum_{y\leq x}a_{y,x}}$$
    subject to the constraints
    \begin{enumerate}
    \item $a_{y,x}\in\mathbb{R}_{\geq0}$;
    \item $A_{s}=s+1$, for all $s\in S$;
    \item $A_xA_y\in \bigoplus_z\mathbb{R}_{\geq0}A_z$, for all $x,y\in W$;
    \item $A_{w_{0}}={\displaystyle \sum_{y\in W}y}$.
    \end{enumerate}
\end{optproblem}

\begin{problem}
   Does the optimization problem  \ref{conj:KLBasisAt1} admit a unique solution? Is this solution the Kazhdan--Lusztig basis?
\end{problem}

\begin{rem}
Note that Problem \ref{conj:KLBasisAt1} is a quadratic optimization problem. The evaluation of the Kazhdan--Lusztig basis at $v=1$ is a feasible solution. In type $A$, the Kazhdan--Lusztig basis is the \emph{unique} feasible solution for $n \leq 6$. In type $A_7$, the specialization at $v=1$ of the $2$-canonical basis provides another feasible solution, which is not a minimizer, see \cite{Williamson2012,JensenWilliamson2017}. This is remarkable, since in other types uniqueness fails much earlier, e.g. for $n=2$ in type $B_n$.
We have verified that the Kazhdan--Lusztig basis at $v=1$ is the unique solution to Problem \ref{conj:KLBasisAt1}, for all finite Coxeter groups with rank $\leq3$. In our experiments, condition (4) is only necessary for dihedral groups $I_{2}(2m)$, but leads to a speed-up for other groups as well. The main bottleneck lies in the large number of quadratic constraints coming from (3) and without further simplifications, we need to consider $\lvert W\rvert^{3}$ variables of the form $\tilde{\mu}_{xy}^{z}$.
\end{rem}
If we drop condition (4) we obtain the following minimal basis, when $W=I_{2}(2m)$.
\begin{rem}
Let $m$ be even and let $W=\langle s,t\mid$$s^{2}=t^{2}=(st)^{m}\rangle$. We define $A_{x}\in \mathbb{R}[W]$ recursively as follows:
\begin{align*}
 & A_{x}=\begin{cases}
A_{xs}s & xs<x\\
sA_{sx} & sx<x\\
A_{t}A_{tx}-A_{stxts} & \ell(x)\geq5\text{ and }xs>x\text{ and }sx>x\\
A_{xt}A_{t}-A_{stxts} & \ell(x)\geq5\text{ and }xs>x\text{ and }sx>x
\end{cases}.
\end{align*}
Where the first elements are given by $A_{s}=s+1,A_{t}=t+1,A_{st}=s\cdot A_{t},A_{ts}=A_{t}\cdot s,A_{sts}=A_{st}\cdot s,A_{tst}=A_{t}A_{st}$.
We can check that the structure constants are non-negative. For example, for $m=6$ we have
\begin{align*}
A_{st}\cdot A_{stst} & =sA_{t}sA_{t}sA_{t}\\
 & =sA_{tsts}A_{t}=s(A_{tstst}+A_{t})=A_{ststst}+A_{st}.
\end{align*}
In Figure \ref{fig:dihedral_matrices}, we see the basis written in the standard basis of $\mathbb{Z}[W]$ in the case $m=6,m=10$ and elements ordered in the form ${\rm id},s,t,st,ts,sts,tst,\dots$ Analogously, we can obtain another basis by exchanging the roles of $s$ and $t$.
\begin{figure}[H]
    \centering
    \begin{subfigure}[t]{0.45\textwidth}
        \centering
        \resizebox{\linewidth}{!}{\begin{tikzpicture}[scale=0.5]
    \definecolor{pink}{rgb}{1,0.7,0.8}
    \definecolor{lightyellow}{rgb}{1,1,0.6}
    \definecolor{orange}{rgb}{1,0.6,0.2}
    \definecolor{lightblue}{rgb}{0.5,1,1}
    \definecolor{blue}{rgb}{0.4,0.8,1}
    \definecolor{aqua}{rgb}{0.4,1,0.8}

    \foreach \y [count=\row from 1] in {
        {1,1,1,0,0,1,0,1,1,0,1,0},
        {0,1,0,1,1,0,1,0,0,1,0,1},
        {0,0,1,0,0,0,0,1,1,0,1,0},
        {0,0,0,1,0,0,1,0,0,1,0,1},
        {0,0,0,0,1,0,1,0,0,1,0,1},
        {0,0,0,0,0,1,0,1,1,0,1,0},
        {0,0,0,0,0,0,1,0,0,0,0,1},
        {0,0,0,0,0,0,0,1,0,0,1,0},
        {0,0,0,0,0,0,0,0,1,0,1,0},
        {0,0,0,0,0,0,0,0,0,1,0,1},
        {0,0,0,0,0,0,0,0,0,0,1,0},
        {0,0,0,0,0,0,0,0,0,0,0,1},
    }{
        \foreach \x [count=\col from 1] in \y {
    \ifnum \row<3
        \ifnum\col<3
            \fill[orange!50] (\col-1,-\row) rectangle (\col,-\row+1); 
        \else
            \ifnum\col<21
                \fill[blue!50] (\col-1,-\row) rectangle (\col,-\row+1); 
            \fi
            \ifnum\col<19
                \fill[green!50] (\col-1,-\row) rectangle (\col,-\row+1); 
            \fi
            \ifnum\col<15
                \fill[blue!50] (\col-1,-\row) rectangle (\col,-\row+1); 
            \fi
            \ifnum\col<11
                \fill[green!50] (\col-1,-\row) rectangle (\col,-\row+1); 
            \fi
            \ifnum\col<7
                \fill[blue!50] (\col-1,-\row) rectangle (\col,-\row+1); 
            \fi
        \fi
    \else
           \ifnum\row<21
                \ifnum\col>18
                    \ifnum\col>2
                        \ifnum\col<21
                            \fill[yellow!50] (\col-1,-\row) rectangle (\col,-\row+1); 
                        \fi
                    \fi
                \fi
            \fi
           \ifnum\row<19
                \ifnum\col>14
                    \ifnum\col>2
                        \ifnum\col<21
                            \fill[blue!50] (\col-1,-\row) rectangle (\col,-\row+1); 
                        \fi
                        \ifnum\col<19
                            \fill[yellow!50] (\col-1,-\row) rectangle (\col,-\row+1); 
                        \fi
                    \fi
                \fi
            \fi
           \ifnum\row<15
                \ifnum\col>10
                    \ifnum\col>2
                        \ifnum\col<21
                            \fill[green!50] (\col-1,-\row) rectangle (\col,-\row+1); 
                        \fi
                        \ifnum\col<19
                            \fill[blue!50] (\col-1,-\row) rectangle (\col,-\row+1); 
                        \fi
                        \ifnum\col<15
                            \fill[yellow!50] (\col-1,-\row) rectangle (\col,-\row+1); 
                        \fi
                    \fi
                \fi
            \fi
           \ifnum\row<11
                \ifnum\col>6
                    \ifnum\col>2
                        \ifnum\col<21
                            \fill[blue!50] (\col-1,-\row) rectangle (\col,-\row+1); 
                        \fi
                        \ifnum\col<19
                            \fill[green!50] (\col-1,-\row) rectangle (\col,-\row+1); 
                        \fi
                        \ifnum\col<15
                            \fill[blue!50] (\col-1,-\row) rectangle (\col,-\row+1); 
                        \fi
                        \ifnum\col<11
                            \fill[yellow!50] (\col-1,-\row) rectangle (\col,-\row+1); 
                        \fi
                    \fi
                \fi
            \fi
           \ifnum\row<7
                \ifnum\col>2
                    \ifnum\col<21
                        \fill[green!50] (\col-1,-\row) rectangle (\col,-\row+1); 
                    \fi
                    \ifnum\col<19
                        \fill[blue!50] (\col-1,-\row) rectangle (\col,-\row+1); 
                    \fi
                    \ifnum\col<15
                        \fill[green!50] (\col-1,-\row) rectangle (\col,-\row+1); 
                    \fi
                    \ifnum\col<11
                        \fill[blue!50] (\col-1,-\row) rectangle (\col,-\row+1); 
                    \fi
                    \ifnum\col<7
                        \fill[yellow!50] (\col-1,-\row) rectangle (\col,-\row+1); 
                    \fi
                \fi
            \fi
        \fi
    \node at (\col-0.5,-\row+0.5) {\x};
}
    }
\end{tikzpicture}}
        \caption{$W = I_2(6)$}
        \label{fig:matrixD6}
    \end{subfigure}
    \hfill
    \begin{subfigure}[t]{0.45\textwidth}
        \centering
        \resizebox{\linewidth}{!}{\begin{tikzpicture}[scale=0.5]
    \definecolor{pink}{rgb}{1,0.7,0.8}
    \definecolor{lightyellow}{rgb}{1,1,0.6}
    \definecolor{orange}{rgb}{1,0.6,0.2}
    \definecolor{lightblue}{rgb}{0.5,1,1}
    \definecolor{blue}{rgb}{0.4,0.8,1}
    \definecolor{aqua}{rgb}{0.4,1,0.8}

    \foreach \y [count=\row from 1] in {
        {1,1,1,0,0,1,0,1,1,0,1,0,0,1,0,1,1,0,1,0},
        {0,1,0,1,1,0,1,0,0,1,0,1,1,0,1,0,0,1,0,1},
        {0,0,1,0,0,0,0,1,1,0,1,0,0,1,0,1,1,0,1,0},
        {0,0,0,1,0,0,1,0,0,1,0,1,1,0,1,0,0,1,0,1},
        {0,0,0,0,1,0,1,0,0,1,0,1,1,0,1,0,0,1,0,1},
        {0,0,0,0,0,1,0,1,1,0,1,0,0,1,0,1,1,0,1,0},
        {0,0,0,0,0,0,1,0,0,0,0,1,1,0,1,0,0,1,0,1},
        {0,0,0,0,0,0,0,1,0,0,1,0,0,1,0,1,1,0,1,0},
        {0,0,0,0,0,0,0,0,1,0,1,0,0,1,0,1,1,0,1,0},
        {0,0,0,0,0,0,0,0,0,1,0,1,1,0,1,0,0,1,0,1},
        {0,0,0,0,0,0,0,0,0,0,1,0,0,0,0,1,1,0,1,0},
        {0,0,0,0,0,0,0,0,0,0,0,1,0,0,1,0,0,1,0,1},
        {0,0,0,0,0,0,0,0,0,0,0,0,1,0,1,0,0,1,0,1},
        {0,0,0,0,0,0,0,0,0,0,0,0,0,1,0,1,1,0,1,0},
        {0,0,0,0,0,0,0,0,0,0,0,0,0,0,1,0,0,0,0,1},
        {0,0,0,0,0,0,0,0,0,0,0,0,0,0,0,1,0,0,1,0},
        {0,0,0,0,0,0,0,0,0,0,0,0,0,0,0,0,1,0,1,0},
        {0,0,0,0,0,0,0,0,0,0,0,0,0,0,0,0,0,1,0,1},
        {0,0,0,0,0,0,0,0,0,0,0,0,0,0,0,0,0,0,1,0},
        {0,0,0,0,0,0,0,0,0,0,0,0,0,0,0,0,0,0,0,1},
    }{
        \foreach \x [count=\col from 1] in \y {
    \ifnum \row<3
        \ifnum\col<3
            \fill[orange!50] (\col-1,-\row) rectangle (\col,-\row+1); 
        \else
            \ifnum\col<21
                \fill[blue!50] (\col-1,-\row) rectangle (\col,-\row+1); 
            \fi
            \ifnum\col<19
                \fill[green!50] (\col-1,-\row) rectangle (\col,-\row+1); 
            \fi
            \ifnum\col<15
                \fill[blue!50] (\col-1,-\row) rectangle (\col,-\row+1); 
            \fi
            \ifnum\col<11
                \fill[green!50] (\col-1,-\row) rectangle (\col,-\row+1); 
            \fi
            \ifnum\col<7
                \fill[blue!50] (\col-1,-\row) rectangle (\col,-\row+1); 
            \fi
        \fi
    \else
           \ifnum\row<21
                \ifnum\col>18
                    \ifnum\col>2
                        \ifnum\col<21
                            \fill[yellow!50] (\col-1,-\row) rectangle (\col,-\row+1); 
                        \fi
                    \fi
                \fi
            \fi
           \ifnum\row<19
                \ifnum\col>14
                    \ifnum\col>2
                        \ifnum\col<21
                            \fill[blue!50] (\col-1,-\row) rectangle (\col,-\row+1); 
                        \fi
                        \ifnum\col<19
                            \fill[yellow!50] (\col-1,-\row) rectangle (\col,-\row+1); 
                        \fi
                    \fi
                \fi
            \fi
           \ifnum\row<15
                \ifnum\col>10
                    \ifnum\col>2
                        \ifnum\col<21
                            \fill[green!50] (\col-1,-\row) rectangle (\col,-\row+1); 
                        \fi
                        \ifnum\col<19
                            \fill[blue!50] (\col-1,-\row) rectangle (\col,-\row+1); 
                        \fi
                        \ifnum\col<15
                            \fill[yellow!50] (\col-1,-\row) rectangle (\col,-\row+1); 
                        \fi
                    \fi
                \fi
            \fi
           \ifnum\row<11
                \ifnum\col>6
                    \ifnum\col>2
                        \ifnum\col<21
                            \fill[blue!50] (\col-1,-\row) rectangle (\col,-\row+1); 
                        \fi
                        \ifnum\col<19
                            \fill[green!50] (\col-1,-\row) rectangle (\col,-\row+1); 
                        \fi
                        \ifnum\col<15
                            \fill[blue!50] (\col-1,-\row) rectangle (\col,-\row+1); 
                        \fi
                        \ifnum\col<11
                            \fill[yellow!50] (\col-1,-\row) rectangle (\col,-\row+1); 
                        \fi
                    \fi
                \fi
            \fi
           \ifnum\row<7
                \ifnum\col>2
                    \ifnum\col<21
                        \fill[green!50] (\col-1,-\row) rectangle (\col,-\row+1); 
                    \fi
                    \ifnum\col<19
                        \fill[blue!50] (\col-1,-\row) rectangle (\col,-\row+1); 
                    \fi
                    \ifnum\col<15
                        \fill[green!50] (\col-1,-\row) rectangle (\col,-\row+1); 
                    \fi
                    \ifnum\col<11
                        \fill[blue!50] (\col-1,-\row) rectangle (\col,-\row+1); 
                    \fi
                    \ifnum\col<7
                        \fill[yellow!50] (\col-1,-\row) rectangle (\col,-\row+1); 
                    \fi
                \fi
            \fi
        \fi
    \node at (\col-0.5,-\row+0.5) {\x};
}
    }
\end{tikzpicture}}
        \caption{$W = I_2(10)$}
        \label{fig:matrixD10}
    \end{subfigure}

    \caption{Optimal solution in the case where $W$ is of type $I_2(6)$ and $I_2(10)$ in Problem~\ref{conj:KLBasisAt1} when dropping condition (4) in Problem \ref{conj:KLBasisAt1}. Blocks of the same form are highlighted in the same colour.}
    \label{fig:dihedral_matrices}
\end{figure}

\end{rem}

We can simplify condition 3 in Problem \ref{conj:KLBasisAt1} by imposing vanishing conditions on the structure constants, for instance $\mu_{x,y}^z=0$ if $\ell(x)+\ell(y)<\ell(z)$, which we know hold for the Kazhdan--Lusztig basis. However, in general the number of necessary variables remains $\mathcal{O}(\lvert W\rvert^3)$.

For type $A$, we can exploit the action of $b_s$ for $s\in S$ on the Kazhdan--Lusztig basis and omit the optimization criteria in the following, which we have checked up to $A_{6}$, i.e. $S_{7}$. 
\begin{optproblem}\label{conj:KLBasisAt1TypeA}
    Let $W$ be of type $A$. Consider all  $A_{x}={\displaystyle \sum_{y\leq x}}a_{y,x}y\in\mathbb{R}[W]$ with
\begin{enumerate}
\item $A_{s}=s+1$ and
\item $A_{s}A_{x}=\begin{cases}
2A_{x}, & sx<x,\\
A_{sx}+\sum_{_{sy<y}^{y<x}}\tilde{\mu}(y,x)A_{y}, & sx>x,
\end{cases}$ \hspace{0.5em} with $\tilde{\mu}(y,x)\in\mathbb{R}_{\geq0}$.
\end{enumerate}
\end{optproblem}

We can further exploit the scarcity of distinct values of the structure constants $\mu$ by introducing more conditions on them, for instance, as in \cite{Warrington2011}.

\section{A Maximal Cone Condition for Kazhdan--Lusztig Left Cell Representations}\label{sec:compute_cones}

While the optimization approach in the group ring works well for small rank, we face a computational barrier at $S_8$. This is mainly due to the large number of structure constants needed to define the basis globally. Therefore, we restrict our attention to specific irreducible representations.

Let $W=S_{n}$ be the symmetric group generated by simple reflections $S$, so $W$ is of Coxeter type $A_{n-1}$, and let $\lambda$ be a partition of $n$. We consider the Specht module $S^{\lambda}$ and the corresponding Kazhdan--Lusztig left cell representation \cite{KazhdanLusztig79}. We aim to recover this representation intrinsically, avoiding expensive computations of Kazhdan--Lusztig polynomials. For this, we recall that in the Kazhdan--Lusztig basis, the matrices of $1+s\in \mathrm {End}(S^\lambda)$ only have non-negative entries, for all $s\in S$. The same is true for Young's seminormal representation, which follows immediately, since every non-diagonal entry is larger than zero and every diagonal entry is larger than $-1$. We summarize these observations in the following lemma.

\begin{lem}
    When expressed in both Young's seminormal basis and Kazhdan--Lusztig basis, the matrices of  $1+s\in \mathrm {End}(S^\lambda)$ have non-negative entries.
\end{lem}

\begin{definition}
    Let $\mathcal{A}$ be a family of endomorphisms of $\mathbb{R}^d$. A set $C \subseteq \mathbb{R}^d$ is called a \textit{cone} if:
    \begin{enumerate}
        \item $C$ is closed in the standard Euclidean topology.
        \item $C$ is closed under taking non-negative linear combinations of its elements.
    \end{enumerate}
    
    Given such a cone $C$, we say that it is an \textit{invariant cone} for the family $\mathcal{A}$ if:
    \begin{equation*}
        A(C) \subseteq C \quad \text{for all } A \in \mathcal{A}.
    \end{equation*}
     A cone $C$ is \emph{pointed} if $C\cap(-C)=\{0\}$ and \emph{solid} (nondegenerate) if it has non-empty interior. A cone is called \emph{proper} if it is pointed, and solid. We say $C$ is polyhedral if $C$ is spanned by finitely many vectors, and we say $C$ is simplicial if it is spanned by exactly $d$ linearly independent vectors.
\end{definition}

The following theorem now follows immediately.

\begin{thm}
Let $\mathcal{A}=\{\,1+s \mid s\in S\,\}\subset \mathrm {End}(S^\lambda)$. Both Young's seminormal basis and the Kazhdan--Lusztig basis span proper $\mathcal{A}$-invariant simplicial cones in $S^\lambda$.
\end{thm}

\begin{rem}
    Let $C \subseteq S^\lambda$ be a cone inside the Specht module for the partition $\lambda$, which is invariant under the operators $\mathcal{A}=\{ 1+s \mid s\in S\}$. Given an $S_n$-invariant form $(\cdot,\cdot)$ on $S^\lambda$, we consider the set
    \[
        C^{*}=\{u \in S^\lambda \mid (u,c)\geq 0 \text{ for all } c \in C\}.
    \]
    We call $C^{*}$ the \emph{dual cone} of $C$. To see that $C^*$ is also invariant under $\mathcal{A}$, let $s\in S$, $u \in C^*$, and let $c \in C$ be arbitrary. Using the self-adjointness of $1+s$ with respect to the invariant form, we have
    \[
        ((1+s)u, c) =(u,c)+(su,c)=(u,c)+(u,sc)= (u, (1+s)c).
    \]
    Since $C$ is invariant under $1+s$, we know $(1+s)c \in C$. By definition of the dual cone, $(u, (1+s)c) \geq 0$. Thus, $((1+s)u, c) \geq 0$ for all $c \in C$, which implies $(1+s)u \in C^*$. Furthermore, if $C$ is the cone generated by a basis $\mathcal{B}$, then $C^*$ is the cone generated by the dual basis $\mathcal{B}^*$ with respect to the form.
\end{rem}

A matrix family $\mathcal{A}$ is called irreducible if we cannot find a base-change such that all matrices in $\mathcal{A}$ have upper block-diagonal form. In \cite{MejstrikProtasov2025}, it is shown that if an irreducible matrix family that has a proper invariant cone has a unique maximal invariant cone.
Since the family of matrices $1+s$ is irreducible and leaves a non-trivial cone invariant, we can thus deduce the following.

\begin{thm}
    The operators $\{1+s \mid s\in S\}$ preserve a unique maximal invariant cone $C^{\max}$ and a unique minimal invariant cone $C^{\min}$ inside $S^\lambda$. Moreover, these cones satisfy the duality relation $(C^{\max})^*=C^{\min}$.
\end{thm}

This naturally leads to the following question.

\begin{problem}
    What are the maximal and minimal $1+s,s\in S$ invariant cones inside the Specht module $S^\lambda$?
\end{problem}
In the following, we show that for hook-shape, two-column and $(n-2,2)$ partitions, the maximal invariant cone is indeed given by the Kazhdan--Lusztig basis.

\begin{example}
Consider the reflection representation, which corresponds to the Specht module for the partition $\lambda=(n,1)$. We consider the Kazhdan--Lusztig basis consisting of simple roots $\alpha_{i}$ and the dual basis consisting of fundamental weights $\varpi_{i}$, for $i=1,\dots,n$, with regard to the usual invariant form as in Example \ref{ex:simple_roots}. Let $F$ be the averaging operator over the operators $1+s$, i.e. $F=\frac{1}{n}\sum_i (1+s_i)$. Then $F$ has a unique maximal eigenvector $v$, which lies in the minimal cone. Now, we can compute for each $j\in \{1,\dots,n \}$
\[
\lim_{k\to\infty}\bigg(\prod_{j\neq i}(1+s_{i})\bigg)^{k}v,
\]
which converges to an eigenvector for the largest eigenvalue of $\prod_{j\neq i}(1+s_{i})$. This basis spans the minimal cone and is given by the fundamental weights $\varpi_i$. Note that one can show that in this case simply iteratively applying the operators $1+s$ does not yield the minimal cone but only an approximation of it.
\end{example}

\begin{rem}
    In general, deciding if a finite set of matrices possesses a proper invariant cone is an undecidable problem \cite{Protasov2010} and even if we know the existence of a minimal/maximal cone it remains a challenge to compute those cones even for small-dimensional matrices, see \cite{MejstrikProtasov2025}. 
\end{rem}

\begin{problem}
Can we compute the maximal $1+s$ invariant cone for $S^\lambda$ using only efficient linear algebraic methods?
\end{problem}


To establish the maximal $(1+s)$-invariant cone property for the Kazhdan--Lusztig left cell representation, we turn to the dual problem. We demonstrate that the dual representation yields the minimal invariant cone. Hence, when writing our matrices in this dual basis, we must prove that the positive orthant $\mathbb{R}_{\geq0}^{d}$ acts as the minimal invariant cone. To do this, we consider a vector $v \in C^{\min}$. We choose $v$ to be the Perron-Frobenius vector of the averaging operator
$$F = \frac{1}{n}\sum_i (1+s_i).$$
This vector lies inside the interior of the minimal cone. Thus, $v$ is strictly positive entrywise in our given basis. Finally, we recall that each basis element $w_{T}$ of the dual Kazhdan--Lusztig basis corresponds to a standard Young tableau $T$ of shape $\lambda$. Moreover, we have that
\[
(1+s)^tw_{T}=2w_{T},
\]
if and only if $s\in D^{c}(T)$. We need the following crucial fact.
\begin{lem}\label{lemma:eigenvectors}
For $T\in\mathrm{Std}(\lambda$), consider $$A^t_{T}=A^t_{s_{i_{1}}}\cdots A^t_{s_{i_{k}}}$$ with $D^{c}(T)=\{i_{1},\dots,i_{k}\}$. The maximal eigenvalue of $A^t_{T}$ is $2^{k}$ and the eigenvectors live in the span
\[
\langle w_{T'}\mid D^{c}(T)\subseteq D^{c}(T')\rangle.
\]
Moreover, for any vector $v$ with positive entries we have that
\[
\frac{A^{n}v}{|A^{n}v|}\to\tilde{v}\in\langle w_{T'}\mid D^{c}(T)\subseteq D^{c}(T')\rangle.
\]
\end{lem}
\begin{proof}
    Clearly all $w_{T'}$ with $D^{c}(T)\subseteq D^{c}(T')$ satisfy $A^t_{T}w_{T'}=2^{k}w_{T'}.$  It suffices to show that these span exactly the eigenspace corresponding to the largest eigenvalue. Consider the operator $A^t_{s_i} = 1 +  s_i^t$. By the triangle inequality, the operator norm satisfies $\|A^t_{s_i}\| = \|1 + s_i^t\| \leq \|1\| + \| s_i^t\| = 2.$
    Equality holds, i.e., $\|A^t_{s_i}v\| = 2\|v\|$, if and only if $ s_i^t v = v$. Now consider $A^t_T = A^t_{s_{i_1}} \cdots A^t_{s_{i_k}}$. Since the operator norm is sub-multiplicative it follows that $$2^k \leq \|A^t_T\| \le \prod_{j=1}^k \|A^t_{s_{i_j}}\| \le 2^k.$$
    Thus, the maximal eigenvalue of $A^t_T$ is given by $2^k$. We now characterize the eigenspace for the value $2^k$. Let $v$ be a vector such that $A^t_T v = 2^k v$. Taking norms:
    \[
    \|A^t_{s_{i_1}} \cdots A^t_{s_{i_k}} v\| = 2^k \|v\|.
    \]
    Since the norm of each factor is at most 2, the equality holds if and only if each operator $A^t_{s_{i_j}}$ acts on the vector $v$ as multiplication by 2, which by definition occurs if and only if $i \in D^c(T')$.  With this the limit property follows using a standard argument.
\end{proof}
If $T$ is already determined by the set $D^{c}(T)=\{1,\dots,n-1\}\setminus D(T)$ and there is no $T'$ with $D^c(T)\subseteq D^c(T')$, Lemma \ref{lemma:eigenvectors} implies that the operator
\[
A^t_{T}=\prod_{s\in D^{c}(T)}(1+s)^t
\]
has $w_{T}$ as unique maximal eigenvector (up to scalar).
In the case of hook-shape partitions, each tableau is uniquely determined by its descent and complement of the descent set and the underlying $W$-graph can be described explicitly, see \cite[Proposition 6.6.1]{BjornerBrenti2005}. Thus, the preceding discussion leads to the following result.

\begin{thm}
Let $\lambda$ be a hook shape partition. Then the Kazhdan--Lusztig left cell representation spans the maximal invariant cone inside the Specht module $S^\lambda$.
\end{thm}
In general, we cannot use this strategy to show that the given cone is minimal, as the operators $A^t_{T}$ can have non-simple Perron-Frobenius eigenspace.

\begin{example}\label{ex:3_2}
Consider the partition $\lambda=(3,2)$ and the Specht module with Kazhdan--Lusztig left cell representation and associated basis elements labelled by the standard tableaux ordered in the last letter order. We can compute their descents and the complements of their descents, see Table \ref{tab:correspondence}. (The remaining entries in this table, such as the cup diagrams, are explained below.)

\begin{table}
\centering
\ytableausetup{centertableaux,notabloids,smalltableaux} 
\begin{tabular}{c|ccccc}
 & $T_{2,4}$ & $T_{3,4}$ & $T_{2,5}$ & $T_{3,5}$ & $T_{4,5}$ \\ \hline
\rule{0pt}{1.5em} 
Standard Tableaux &
\begin{ytableau}
1 & 3 & 5 \\
2 & 4
\end{ytableau} &
\begin{ytableau}
1 & 2 & 5 \\
3 & 4
\end{ytableau} &
\begin{ytableau}
1 & 3 & 4 \\
2 & 5
\end{ytableau} &
\begin{ytableau}
1 & 2 & 4 \\
3 & 5
\end{ytableau} &
\begin{ytableau}
1 & 2 & 3 \\
4 & 5
\end{ytableau}
\\[1.5em] 

Cup Diagrams &
\begin{tikzpicture}[scale=0.3, baseline=5pt, thick]
    \foreach \x in {1,...,5} \filldraw (\x,0) circle (2pt);
    \draw (1,0) arc (180:0:0.5);
    \draw (3,0) arc (180:0:0.5);
    \draw (5,0) -- (5,1.5);
\end{tikzpicture} &
\begin{tikzpicture}[scale=0.3, baseline=5pt, thick]
    \foreach \x in {1,...,5} \filldraw (\x,0) circle (2pt);
    \draw (2,0) arc (180:0:0.5);
    \draw (1,0) arc (180:0:1.5);
    \draw (5,0) -- (5,1.5);
\end{tikzpicture} &
\begin{tikzpicture}[scale=0.3, baseline=5pt, thick]
    \foreach \x in {1,...,5} \filldraw (\x,0) circle (2pt);
    \draw (1,0) arc (180:0:0.5);
    \draw (4,0) arc (180:0:0.5);
    \draw (3,0) -- (3,1.5);
\end{tikzpicture} &
\begin{tikzpicture}[scale=0.3, baseline=5pt, thick]
    \foreach \x in {1,...,5} \filldraw (\x,0) circle (2pt);
    \draw (2,0) arc (180:0:0.5);
    \draw (4,0) arc (180:0:0.5);
    \draw (1,0) -- (1,1.5);
\end{tikzpicture} &
\begin{tikzpicture}[scale=0.3, baseline=5pt, thick]
    \foreach \x in {1,...,5} \filldraw (\x,0) circle (2pt);
    \draw (3,0) arc (180:0:0.5);
    \draw (2,0) arc (180:0:1.5);
    \draw (1,0) -- (1,1.5);
\end{tikzpicture}
\\[1.5em] 

$D(T)$ & $\{1,3\}$ & $\{2\}$ & $\{1,4\}$ & $\{2,4\}$ & $\{3\}$ \\[1.5em]

$D(T)^{c}$ & $\{2,4\}$ & $\{1,3,4\}$ & $\{2,3\}$ & $\{1,3\}$ & $\{1,2,4\}$
\end{tabular}
\caption{Standard tableaux, cup diagrams, and descent sets and their complements for the partition $\lambda=(3,2)$. See below for an explanation of the cup diagrams.}
\label{tab:correspondence}
\end{table}

Starting with a vector $v$ in the interior of the minimal cone, we can obtain all $w_T$ with $T$ maximal with regard to the poset defined by the inclusion of the sets $D^c(T)$. Thus, we obtain the basis vectors for $T_{3,4},T_{2,5},T_{4,5}$. In the following, we show how, starting with the basis vector of $T_{3,4}$, we can obtain all other basis vectors. Using the $W$-graph $\Gamma^{(3,2)}$ in Figure \ref{fig:3_2}, we can compute $A^t_{s_2}w_{T_{3,4}}=w_{T_{2,4}}$, so we obtain the basis vector for $T_{2,4}$. Next, we check that $A^t_{s_3}w_{T_{2,4}}=w_{T_{3,4}}+w_{T_{2,5}}=v_1$.  Since $w_{T_{2,5}}$ is the only largest eigenvector (up to scalar) of the matrix $A=A^t_{T_{2,5}}=A^t_{s_2}A^t_{s_3}$,  we can check that $A^nv/|A^nv|\to w_{T_{2,5}}$. We further compute $A^t_{s_1}w_{T_{2,5}}=w_{T_{3,5}}$ and obtain $w_{T_{4,5}}$ similarly to $w_{T_{2,5}}$.

In Figure \ref{fig:5_2}, we have visualised a similar process in the case  $\lambda=(5,2)$, which contains $(3,2)$ as a subcase. Note that in this case $T_{2,5}$ no longer is a maximal tableau with regard to inclusion of the sets $D^c(\cdot)$.
\begin{figure}
    \centering
    \resizebox{.5\linewidth}{!}{\begin{tikzpicture}[
    vertex/.style={circle, draw, font=\sffamily\huge\bfseries, minimum size=16mm, inner sep=0pt},
    source/.style={vertex, fill=red, text=white},         
    sink/.style={vertex, fill=green!60!black, text=white}, 
    elabel/.style={blue, font=\sffamily\huge\bfseries, minimum size=8mm, sloped, allow upside down=false},
]

\def\R{7}                
\def\RedList{1,4,8,9,13} 
\def\TotalNodes{14}      

\foreach \i/\j [count=\n from 0] in {
    2/4, 
    3/4, 
    2/5, 
    3/5, 
    4/5, 
    2/6, 
    3/6, 
    4/6, 
    5/6, 
    2/7, 
    3/7, 
    4/7, 
    5/7, 
    6/7  
} {
    \pgfmathsetmacro{\ang}{90 + \n * (360/\TotalNodes)}

    \def\CurrentStyle{sink} 
    \foreach \r in \RedList {
        \ifnum\n=\r
            \xdef\CurrentStyle{source} 
            \breakforeach
        \fi
    }

    \node[\CurrentStyle] (\n) at (\ang:\R) {$T_{\i,\j}$};
}
\draw[-{Stealth[length=5mm]}] (1) to node[elabel,above] {$s_2$} (0);
\draw[-{Stealth[length=5mm]}] (0) to node[elabel,below] {$s_3$} (2);
\draw[-{Stealth[length=5mm]}] (2) to node[elabel,below] {$s_1$} (3);
\draw[-{Stealth[length=5mm]}] (3) to node[elabel,below] {$s_2$} (4);

\draw[-{Stealth[length=5mm]}] (2) to node[elabel,above] {$s_4$} (5);
\draw[-{Stealth[length=5mm]}] (5) to node[elabel,below] {$s_1$} (6);
\draw[-{Stealth[length=5mm]}] (6) to node[elabel,below] {$s_2$} (7);
\draw[-{Stealth[length=5mm]}] (7) to node[elabel,below] {$s_3$} (8);

\draw[-{Stealth[length=5mm]}] (5) to node[elabel,above] {$s_5$} (9);
\draw[-{Stealth[length=5mm]}] (9) to node[elabel,below] {$s_1$} (10);
\draw[-{Stealth[length=5mm]}] (10) to node[elabel,below] {$s_2$} (11);
\draw[-{Stealth[length=5mm]}] (11) to node[elabel,below] {$s_3$} (12);
\draw[-{Stealth[length=5mm]}] (12) to node[elabel,below] {$s_4$} (13);

\end{tikzpicture}}
    \caption{Strategy for showing that the dual Kazhdan--Lusztig left cell basis spans the minimal $1+s,s\in S$ invariant cone inside $S^\lambda$, for $\lambda=(5,2)$. Tableaux that are maximal with regard to the complement of their descent set are marked in red.}
    \label{fig:5_2}
\end{figure}
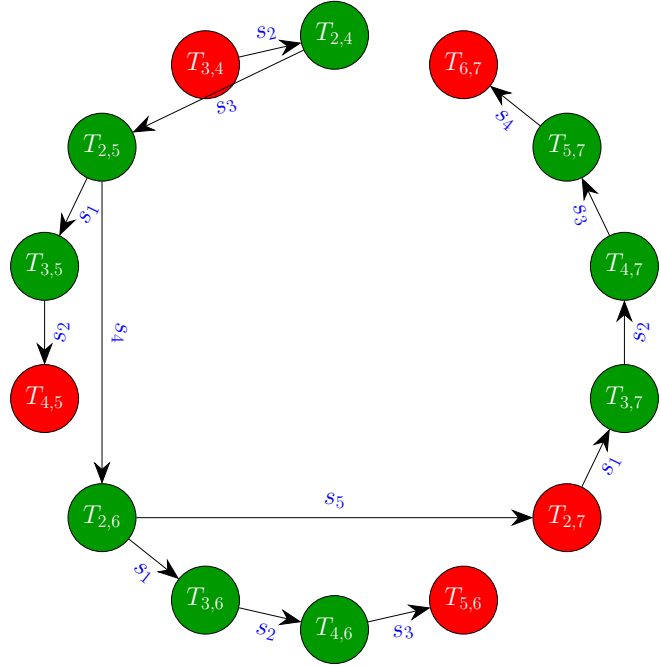
\end{example}

The example suggests the following strategy for verifying the minimal invariant cone property for the dual Kazhdan--Lusztig left cell representation with basis elements of the form $w_T$:
\begin{enumerate}
\item Obtain $w_{T}$ for $T$ with $D^{c}(T)$ maximal, i.e. for all $T'\neq T,$ we have $D^{c}(T)\not\subseteq D^{c}(T')$.
\item For $w_T$ known and $s\in S$
\begin{enumerate}
    \item If $(1+s)^t w_T=w_{T'}$, we have shown $w_{T'}$ lies in the minimal cone.
    \item If $(1+s)^tw_{T}=v$ converges to $w_{T'}$ by applying $A^t_{T'}$ iteratively, i.e.
\end{enumerate}
\[
\frac{A^t_{T'}v}{\lvert A^t_{T'}v\rvert}v\to w_{T'},
\]we have shown $w_{T'}$ lies in the minimal cone.
\item Iterate step (2) until all $w_{T'}$ are obtained.
\end{enumerate}

For two column partitions, this strategy becomes particularly applicable.

\begin{lem}\label{lemma:two_column}
    Let $\lambda$ be a two-column partition and let $A^t_s$ denote the matrix of $(1+s)^t$, for $s\in S$, in the Kazhdan--Lusztig basis. Then, each column of $A^t_s$ has at most one non-zero entry.
\end{lem}

\begin{proof}
    Let $\mu=\lambda'$ be the conjugate partition associated to $\lambda$. Hence, $\mu$ is a two-row partition. Let $\Gamma^\mu$, with vertex labelling $I$, denote the $W$-graph of the Kazhdan--Lusztig left cell representation corresponding to $\mu$. In \cite{Westbury95}, it is shown using diagrammatic calculus that each vertex $x$ with $s\notin I(x)$ in the underlying $W$-graph is connected to at most one vertex $y$ with $s\in I(y)$. Moreover, in \cite{KazhdanLusztig79}, it is shown that the $W$-graph $\Gamma^\lambda$ for $\lambda$ is obtained from $\Gamma^\mu$ by retaining the edge set and taking the complements of the descent sets, i.e., setting $I'(x)=I^c(x)$.
    Now, let $x, y$ be arbitrary vertices of $\Gamma^\lambda$, and let $A^t_s$ be the matrix of $(1+s)^t$. Observe that $y$ appears with a non-zero coefficient in $A^t_s x$ if and only if $x$ appears with a non-zero coefficient in $A_s y$, which occurs if and only if $x$ and $y$ are connected in the $W$-graph $\Gamma^\lambda$. The claim then follows immediately from the connectivity property of $\Gamma^\mu$ described above.
\end{proof}

\begin{thm}\label{thm:two_column_cone}
    Let $\lambda$ be a two-column partition. The Kazhdan--Lusztig basis of $S^\lambda$ spans the maximal invariant cone inside $S^\lambda$.
\end{thm}

\begin{proof}
    Lemma \ref{lemma:eigenvectors} together with step (2)(a) in our strategy imply that the vectors $v$ in our strategy are already of the form $w_{T'}$. By the connectivity of the underlying $W$-graph $\Gamma^\lambda$ of the Kazhdan--Lusztig  left cell representation associated to $\lambda$, we can find for each $T,T'$ reflections $s_1,\dots,s_k$ (not necessarily distinct) with $w_{T'}=A^t_{s_k}\cdots A^t_{s_1}w_T$, i.e.\ iteratively applying step (2)(a) in our strategy. Since there always exists at least one maximal tableau $T$ with regard to step (1) in our strategy, we obtain all elements.
\end{proof}

\begin{rem}
    From Theorem \ref{thm:two_column_cone} we cannot deduce the same property for two-row partitions automatically, and we need to apply the more general case (2)(b) of our strategy. In fact, in general the maximal condition can hold for a partition but fail for its conjugate partition, see Example \ref{ex:cone_counterexample}.
\end{rem}

We demonstrate our general strategy for the case $\lambda=(n-2,2)$ with $n>2$, viewed as a special case of two-row partitions. We briefly recall the combinatorics of obtaining the Kazhdan--Lusztig left cell representation in the $(n-2,2)$ case. For the general case of Kazhdan--Lusztig left cell representations associated to two-row and two-column partitions, several explicit constructions are available. We refer to \cite{LascouxSchuetzenberger} for an approach using words in two letters, to \cite{Kerov83} for a formulation via Young tableaux, and to \cite{Westbury95} for the diagrammatic calculus using Temperley-Lieb algebras. Each standard tableau $T$ is uniquely determined by the entries in its second row say, $2\leq i<j\leq n$ and we write $T_{i,j}=T$. Another diagrammatic way of identifying this tableau is by using cup-diagrams: 
Draw $n$ dots and connect $i-1,i$ and $j-1,j$ by cups if $i+1\neq j$ or dots $i-1,i$ and $j-3,j$ if $i+1=j$, for the remaining dots draw vertical strands.

From this, we can read off the descents of $T$ by considering all $i$ such that $i$ and $i+1$ are connected by a cup. Thus, we have
\[
D(T_{i,j})=\begin{cases}
\{i-1,j-1\}, & i+1\neq j,\\
\{i-1\}, & i+1=j.
\end{cases}
\]
 and hence
\[
D^{c}(T_{i,j})=\begin{cases}
\{1,\dots,i-2,i,\dots,j-2,j,\dots,n-1\}, & i+1\neq j,\\
\{1,\dots,i-2,i,\dots,n-1\}, & i+1=j.
\end{cases}
\]
It follows that the $T_{i,j}$, which are maximal with regard to the $D^c(\cdot)$-inclusion poset are precisely of the form $T_{i,i+1}$ and $T_{2,n}$, see for example Figure \ref{fig:5_2}.

The action of $S_{n}$ is determined by the action of simple reflections on this basis via skein relations. This can be realised combinatorially as follows: $T,T'$ are connected by an edge, denoted by $T\sim T'$, if and only if flipping the diagram of $T'$ and sticking it to the bottom of $T$ leads to exactly one closed circle and $n-4$ through strands. Then, we have for $s_{i}\in S$
\[
s_{i}T=\begin{cases}
-T, & i\in D(T),\\
T+\sum_{T\sim T',i\in D(T')}T', & i\notin D(T).
\end{cases}
\]
As noted in the proof of Lemma \ref{lemma:two_column}, this sum may indeed be empty and only at most one $T'$ can appear. 
In order to verify the maximal cone property, we list some properties needed for our strategy.

\begin{lem}\label{lemma:two_row_computations} Consider $\lambda=(n-2,2)$ for $n>2$ and standard Young tableaux $T_{i,j}$ with $2\leq i<j\leq n$.
\begin{enumerate}
\item $T_{2,4}=(1+s_{2})^tT_{3,4}$.
\item $(1+s_{i-1})^tT_{2,i}=\begin{cases}
T_{2,i-1}+T_{2,i+1} & i>4\\
T_{2,5}+T_{3,4} & i=4
\end{cases}$;
\item For $v=(1+s_{i-1})^tT_{2,i}$ and $A=A^t_{T_{2,i+1}}=\prod_{i\in D^{c}(T_{2,i+1})}(1+s_{i})^t$,
we have $\frac{A^{n}v}{\lvert A^{n}v\rvert}\to T_{2,i+1}$;
\item $(1+s_{i-1})^tT_{i,j}=\begin{cases}
T_{i+1,j}+T_{i-1,j}, & i\neq2,j-2,\\
T_{i+1,j} & i=2\\
T_{i+1,j}+T_{i-1,j}+T_{i+2,j+1} & i=j-2
\end{cases}$;
\item For $v=T_{i,j}(1+s_{i-1})$ and $A=A^t_{T_{i+1,j}}=\prod_{i\in D^{c}(T_{i+1,j})}(1+s_{i})^t$,
we have $\frac{A^{n}v}{\lvert A^{n}v\rvert}\to T_{i+1,j}$.
\end{enumerate}    
\end{lem}

\begin{proof}
    (1)(2)(4) Follow straight from the construction of the $W$-graph. For (3) we note that by applying $A$ to $v$ only terms of the form $T_{2,j}$ appear with $j\leq i+1$. The largest eigenvectors of $A$ live in the span of $T_{2,i+1}$ or $T_{i+1,i+2}$. As the latter term cannot appear by repeatedly applying $A$, we arrive at the result.
    For (5), we can assume that $i\neq 2,j-2$ since for $i=2$ there is nothing to show and in the case $i=j-2$, $T_{i+1,j}=T_{j-1,j}$ is already the unique largest eigenvector of $A$.  Now, the only summands appearing in repeatedly applying $A$ are of the form $T_{k,j}$, where $j$ is fixed. Moreover, the largest eigenvectors of $A$ all live in the span of $T_{i+1,j},T_{i+1,i+2},T_{j,j+1}$ and thus the result follows.
\end{proof}

Using the various computational results from Lemma \ref{lemma:two_row_computations}, we arrive at the following result.

\begin{thm}
    Let $\lambda=(n-2,2)$. The Kazhdan--Lusztig basis of $S^\lambda$ spans the maximal invariant cone inside $S^\lambda$.
\end{thm}

\begin{proof}
    We show that we can arrive at each tableau using the formulas in Lemma \ref{lemma:two_row_computations}(1). Since we start with a vector in the interior of the minimal cone and $T_{3,4}$ is maximal with regard to the complement of its descent set (see Example \ref{ex:3_2}), we can arrive at $T_{2,4}$ (Lemma \ref{lemma:two_row_computations}(2,3)).  Starting from $T_{2,4}$, we arrive at any element of the form $T_{2,i}$ (Lemma \ref{lemma:two_row_computations}(4,5)).
\end{proof}

\begin{rem}
Alternatively, one can also start with $T_{n-1,n}$, which is the maximal tableau in the last letter order and consider $s$ in our strategy such that $w_{T'}$ appears in $w_T$ for $T'<T$. We have verified computationally that this approach indeed works for all partitions $\lambda \vdash n$ with $n\leq 10$ and $\lambda$ not containing the partition $(4,4,1)$, see Example \ref{ex:cone_counterexample} below. We have further verified that the dual Kazhdan--Lusztig left cell representation associated to $\lambda=(3,3,3,3,3)\vdash15$ satisfies the minimal cone condition. 
\end{rem}

Inclusion of invariant cones now leads immediately to the following interesting result.

\begin{cor}
    For $\lambda$ being a hook-shape partition, two-column partition or a partition of the form $\lambda=(n-2,2)$, the base-change matrix between Young's seminormal form and the Kazhdan--Lusztig basis has only non-negative entries.
\end{cor}

\begin{problem}
    Is the base-change matrix between Young's seminormal form and the Kazhdan--Lusztig basis always non-negative (or equivalently is the cone corresponding to  Young's seminormal form always contained in the cone spanned by the Kazhdan--Lusztig basis)?
\end{problem}

This question and the investigations of this section link naturally to a conjecture by Blasiak on the base-change matrix between Hecke algebra analogues of these bases \cite[Conjecture 7.1]{Blasiak2014}. 

\begin{example}\label{ex:cone_counterexample}
    Consider the partition $\lambda=(4,4,1)$ and the Specht module $S^\lambda$ of dimension $84$. By exhausting all possibilities in our strategy, we can check computationally  that it fails to verify the minimal cone property for the dual Kazhdan--Lusztig basis. Indeed, the only basis vector that cannot be reached corresponds to the tableau
       $$\ytableausetup{centertableaux,notabloids,nosmalltableaux}
       T=\begin{ytableau}
1 & 3 & 4 & 8 \\
2 & 5 & 7 & 9 \\
6
\end{ytableau}.$$
     If we enumerate the dual Kazhdan--Lusztig basis elements by $w_1,\dots,w_{84}$ in the last letter order, tableau $T$ corresponds to $w_{14}$ and we can check that the cone spanned by $w_1,\dots,w_{13},w_{15},\dots,w_{84}$ and $w_{14}+w_{44},w_{14}+w_{84}$ is invariant under the operators $1+s$.
    For the conjugate partition $\lambda'=(3,2,2,2)$, we can check computationally that our strategy recovers all basis vectors of the dual Kazhdan--Lusztig basis and thus verifies the minimal cone property for it. We have verified that this is indeed the first case, where this phenomenon occurs, see \cite{Goertzen_kl-optimizer}. Moreover, all other tested examples in which the minimal cone property fails come from partitions, where the corresponding Young diagram contains the Young diagram of the partition $(4,4,1)$.
\end{example}

While the minimal invariant cone is unique, it is not, in general, simplicial. Hence, there may exist multiple minimal invariant simplicial cones containing it. This observation raises the following question.

\begin{problem}
    Among all proper invariant simplicial cones, can we distinguish the one corresponding to the Kazhdan--Lusztig left cell representation?
\end{problem}

This motivates the geometric framework in the following section focused on simplicial cones, with two examples provided by Young's seminormal representation and the Kazhdan--Lusztig left cell representation.

\section{An Optimization Problem for Computing Kazhdan--Lusztig Left Cell Representations}\label{sec:compute_left_cells}

Let $W=S_n$ for a positive integer $n>0$ and $S$ the set of simple reflections inside $S_n$ and $\lambda\vdash n$ a partition of $n$.  
As noted in the previous section, the matrices of the operators $1+s\in\mathrm{End}(S^\lambda)$ written in Young's seminormal form or Kazhdan--Lusztig left cell representation have non-negative entries. These bases share another key property.

\begin{thm}\label{thm:uni_upper_specht}
    The base-change matrix from Young's seminormal basis or the Kazhdan--Lusztig left cell representation basis to the polytabloid basis is upper unitriangular (when ordering the basis vectors with regard to the last letter order).
\end{thm}
\begin{proof}
    For a proof for Young's seminormal basis, see \cite{Murphy81}, and for the Kazhdan--Lusztig basis, see \cite{GarsiaMcLarnan88,Naruse89}.
\end{proof}

\begin{rem}\label{rem:transition_matrix_non_negative}
    The Specht module has an analogue with action of the Hecke algebra, see \cite{DipperJames86}. In \cite{McDonoughPallikaros2005}, it is shown that the transition matrix between the basis of the Specht module of the Hecke algebra and Kazhdan--Lusztig left cell representation is unitriangular with entries given by parabolic Kazhdan--Lusztig polynomials. This was also found independently in \cite{Naruse2025SpechtLeftCell}. It follows, by evaluating these polynomials at $1$, that the transition matrix in the symmetric group case has only non-negative entries since these polynomials are known to have non-negative coefficients, see for instance \cite{KashiwaraTanisaki}. This also shows that each basis element of the polytabloid basis already lives in the invariant cone spanned by the Kazhdan--Lusztig basis.
    Starting with Young's natural representation instead, as described in \cite{GarsiaMcLarnan88}, the base-change matrix to the Kazhdan--Lusztig left cell representation contains negative entries in general. Here, the first occurrence of a base-change matrix with negative entries occurs for the partition $\lambda=(2,2,1,1,1)\vdash7$.
\end{rem}

We want to distinguish the Kazhdan--Lusztig basis intrinsically in $S^\lambda$ using only the $S_n$--action
(at $v=1$) together with the (unique up to positive scalar) invariant positive-definite bilinear form $(\cdot,\cdot)$. In light of Theorem~\ref{thm:uni_upper_specht},
we restrict our search to upper unitriangular base-change matrices $A$ from the polytabloid basis. One
motivation is that any change-of-basis matrix admits a QR factorization into an orthogonal factor and an
upper triangular factor: the orthogonal factor is an isometry for $(\cdot,\cdot)$ and therefore does not affect the
geometry, whereas the triangular factor adds the main distortion. Requiring $A$ to be unitriangular
fixes the volume normalization. As observed in Section~\ref{sec:compute_cones}, the minimal or maximal invariant cones inside $S^\lambda$ need not be
simplicial. Nevertheless, we restrict attention to \emph{invariant simplicial} cones, since these are exactly the
cones arising as a positive orthant in some basis. Concretely, writing $A_s$ for the matrix of $1+s$ in the
polytabloid basis, the condition that $1+s$ has non-negative matrix entries in the $A$-basis is
\[
A^{-1}A_sA \ge 0 \qquad (s\in S).
\]
We seek a geometric criterion that selects a distinguished cone within this restricted family. Let $G$ be the
Gram matrix of $(\cdot,\cdot)$ (the unique up-to-scalar invariant form on $S^\lambda$) in the polytabloid basis and set $G_A := A^{\mathsf t}GA$. For a basis
$v_1,\dots,v_d$ spanning a simplicial cone, define the (squared) normalized volume
\[
\operatorname{nvol}(A)^2 := \frac{\det(G_A)}{\prod_{i=1}^d (G_A)_{ii}}.
\]
By Hadamard's inequality, $\operatorname{nvol}(A)^2\le 1$, with equality precisely when the basis vectors are
pairwise orthogonal. Since $A$ is unitriangular, we have $\det(G_A)=\det(G)$, so extremizing
$\operatorname{nvol}(A)$ is equivalent to extremizing $\prod_i (G_A)_{ii}$. Moreover, by the inequality of arithmetic and geometric means we have
\[
\prod_{i=1}^d (G_A)_{ii} \le \left(\frac{\mathrm{Tr}(G_A)}{d}\right)^d.
\]
Hence, extremizing $\mathrm{Tr}(A^{\mathsf t}GA)$ provides a computationally tractable objective for extremizing
$\prod_i (G_A)_{ii}$ (equivalently, for extremizing $\operatorname{nvol}(A)$) on the feasible region. This
motivates the following optimization problem, in which we consider both maximization and minimization.

\begin{optproblem}\label{conj:KLLeftcellTypeA}
Consider the problem of maximizing or minimizing
\[
f(A)=\mathrm{Tr}(A^tGA)
\]
subject to the constraints
\begin{enumerate}
    \item $A$ is upper unitriangular and
    \item $A^{-1}A_sA\geq0 \phantom{123}(*)$.
\end{enumerate}
Let $A_{\max}$ and $A_{\min}$ be maximizers and minimizers of $f$ under these constraints, respectively. 
\end{optproblem}

In order to show that $A_{\max}$ is well-defined, we first show that the constraints correspond to a compact semialgebraic set. Here compactness is understood with respect to the matrix norm induced by the standard Euclidean norm.

\begin{restatable}{prop}{CompactProp}\label{prop:compact}
Let $A_1,\dots,A_r\in \mathbb{R}^{d\times d}$ be an irreducible family of matrices, i.e.\ they share no non-trivial common invariant subspace. Then
$$\mathcal{F}=\{A\in \mathbb{R}^{d\times d}\mid A \text{ is upper unitriangular and } A^{-1}A_kA\ge 0,\ k=1,\dots,r\}$$
is a compact semialgebraic set.
\end{restatable}

We defer a detailed proof of Proposition \ref{prop:compact} to Appendix \ref{appendix}, as it involves methods different from the previous ones, and might disrupt the flow of the paper.

Note that in the situation of Problem \ref{conj:KLLeftcellTypeA} the feasibility region $\mathcal{F}$ is non-empty. We can define a natural order on $\mathcal{F}$ via containment of invariant cones. First note that for $A\in \mathcal{F}$, $C=A\mathbb{R}_{\geq 0}^d=\{Ax\mid x\in \mathbb{R}_{\geq 0}^d\}$ is a proper invariant cone. In particular, we can view $\mathcal{F}$ as the space of all invariant simplicial cones, spanned by bases that are upper unitriangular with respect to a given basis.
For $A,A'\in \mathcal{F}$, we write $$A\preccurlyeq A' \text{ if } A\mathbb{R}_{\geq 0}^d\subseteq A'\mathbb{R}_{\geq 0}^d,$$ which holds if and only if $A'^{-1}A\geq0$. Proposition \ref{prop:compact} gives rise to several questions regarding the geometry of the feasibility region $\mathcal{F}$ in Problem \ref{conj:KLLeftcellTypeA}. 

\begin{problem}
    What can be said about the geometry of $\mathcal{F}$? What is its dimension? Is it always connected? Does it have a unique minimum and a unique maximum with respect to $\preccurlyeq$?
\end{problem}

We can interpret feasible solutions in Problem \ref{conj:KLLeftcellTypeA} as base-change matrices from the polytabloid basis to a basis such that the operators $1+s$ yield a proper invariant cone. As noted before, both the Kazhdan--Lusztig left cell representation and Young's seminormal form provide feasible solutions. In Theorem \ref{thm:seminormal}, we show that $A_{\min}$ yields Young's seminormal form. 

\begin{problem}
    Which basis corresponds to $A_{\max}$  in Problem \ref{conj:KLLeftcellTypeA}?
\end{problem}

\begin{rem}\label{rem:cell_experiments} 
For all partitions of $n \le 7$, we have verified that the maximum attained in Problem \ref{conj:KLLeftcellTypeA} is unique and yields the base-change matrix to the Kazhdan--Lusztig basis. However, for larger ranks, the situation changes. The action of $1+s$ is non-negative in both the Springer basis and the $p$-canonical bases of Specht modules. See for example \cite{Hotta81,LusztigComments2017} for the Springer basis and \cite{Jensen2020} for $p$-canonical bases. Therefore, these bases give rise to invariant cones. The cone spanned by the Kazhdan--Lusztig basis lies within the cones spanned by the Springer basis and the $p$-canonical bases \cite{Williamson2012,Williamson2015}. Both of these cones yield strictly larger objective values in our trace maximization problem. For examples where these bases differ from the Kazhdan--Lusztig basis, see \cite{Williamson2015} for a case with $\lambda=(4,4,2,2)\vdash 12$, and \cite{LaniniMcNamara2021} for further examples. Following the work of Nguyen \cite{Nguyen2020}, it is possible to distinguish the Kazhdan--Lusztig basis from a $p$-canonical basis by adding further combinatorial constraints  to the optimization problem, such as those defining admissible W-graphs.
In all computed examples, we observed that the Gram matrix of the dual Kazhdan--Lusztig left cell representation is strictly positive. Consequently, the underlying cone is acute. This geometric structure connects naturally to the minimality properties observed in Section \ref{sec:compute_cones}. 
\end{rem}

\begin{rem}
    For practical reasons, we can rewrite $(*)$ as $A_sA=AP_s$, where $P_s$ is an unknown matrix with non-negative entries. Hence, Problem \ref{conj:KLLeftcellTypeA} yields a continuous quadratic optimization problem. A remarkable side effect of Problem \ref{conj:KLLeftcellTypeA} is that we obtain base-change matrices as well as representation matrices of $1+s$. 
    In our experiments, the main bottleneck is the large number of unknowns of the non-negative matrices, $P_s$, which are known to be sparse in practice for both the Kazhdan--Lusztig left cell basis and Young's seminormal form. Further improvement of the performance could be achieved by imposing several conditions on $P_s$, resembling the sparsity of the Kazhdan--Lusztig basis, which can be easily implemented.
\end{rem}

\begin{example}
    Consider the partition $\lambda=(2,1)$. We first fix the standard polytabloid basis $\{v_1, v_2\}$, which is given by:
    \[
    \ytableausetup{boxsize=normal,tabloids}
    v_1 = \ytableaushort{12, 3}-\ytableaushort{23, 1}, \quad
    v_2 = \ytableaushort{13, 2}-\ytableaushort{23, 1}.
    \]
    Using the standard invariant form on tabloids, we compute the Gram matrix in this basis to be:
    \[
        G=\begin{pmatrix}2 & 1\\ 1 & 2 \end{pmatrix}.
    \]
    The representation matrices for the operators $1+s_{1}$ and $1+s_{2}$ in this basis are:
    \[
        A_{s_1} = \begin{pmatrix}0 & -1\\ 0 & 2 \end{pmatrix}, \quad 
        A_{s_2} = \begin{pmatrix}1 & 1\\ 1 & 1 \end{pmatrix}.
    \]
    
    Now, consider a general unitriangular base-change matrix $A = \begin{pmatrix}1 & x\\ 0 & 1 \end{pmatrix}$. We compute the action of the operators in the new basis:
    \begin{align*}
        A^{-1}A_{s_1}A &= \begin{pmatrix}0 & -2x-1\\ 0 & 2 \end{pmatrix}, \\
        A^{-1}A_{s_2}A &= \begin{pmatrix}1-x & 1-x^{2}\\ 1 & 1+x \end{pmatrix}.
    \end{align*}
    Requiring the resulting matrices to be non-negative enforces the constraint $x \in [-1, -0.5]$.
    
    We now consider the optimization problem. The objective function is the trace of the transformed Gram matrix:
    \[
        f(A) = \mathrm{Tr}(A^t G A) = 4+2x+2x^2.
    \]
    We optimize this function over the feasible interval $x \in [-1, -0.5]$:
    \begin{enumerate}
        \item The maximum is attained at $x=-1$. This yields the Kazhdan--Lusztig left cell representation, with operator matrices:
        \[
        \begin{pmatrix}0 & 1\\ 0 & 2 \end{pmatrix}, \quad \begin{pmatrix}2 & 0\\ 1 & 0 \end{pmatrix}.
        \]
        Geometrically (as in Example \ref{ex:simple_roots}), this corresponds to the simple root basis, meaning the polytabloid basis vectors were of the form $\alpha$ and $\alpha+\beta$.
        
        \item The minimum is attained at $x=-0.5$. This yields Young's seminormal form, with operator matrices:
        \[
        \begin{pmatrix}0 & 0\\ 0 & 2 \end{pmatrix}, \quad \begin{pmatrix}0.5 & 0.75\\ 1 & 1.5 \end{pmatrix}.
        \]
        At this minimum, $A$ is the unique upper unitriangular matrix that diagonalizes $G$.
    \end{enumerate}
    The geometric relationship between these bases is visualized in Figure \ref{fig:a2_example}.
    
    \begin{figure}[H]
        \centering
        \begin{tikzpicture}[scale=4.5]
\fill [green!20, opacity=0.5] (0,0) -- (0:0.8) arc (0:109:0.8) -- cycle;
\fill [blue!20, opacity=0.5] (0,0) -- (30:0.6) arc (30:90:0.6) -- cycle;

\draw [->, red, thick] (0,0) -- (1,0) node [right] {$\alpha$};
\draw [->, red, thick] (0,0) -- (-0.5, {sqrt(3)/2}) node [left] {$\beta$};
\draw [->, red, thick] (0,0) -- (0.5, {sqrt(3)/2}) node [right] {$\alpha + \beta$};

\draw [->, blue, thick] (0,0) -- (0.5, {sqrt(3)/6}) node [right] {$\omega_1$};
\draw [->, blue, thick] (0,0) -- (0, {sqrt(3)/3}) node [above] {$\omega_2$};

\draw [->, black, thick]  (0,0) -- (-0.3, {sqrt(3)/2}) node [above right] {$(1+x)\alpha + \beta$};

\end{tikzpicture}
        \caption{Invariant cones for $x\in [-1,-0.5]$ under operators $1+s_1,1+s_2$ for the Specht module $S^{(2,1)}$. The minimal cone given by the fundamental weight basis is also highlighted. The maximal cone is attained for $x=-1$ and agrees with the Kazhdan--Lusztig basis. For $x=-0.5$, we obtain Young's seminormal basis.}
        \label{fig:a2_example}
    \end{figure}
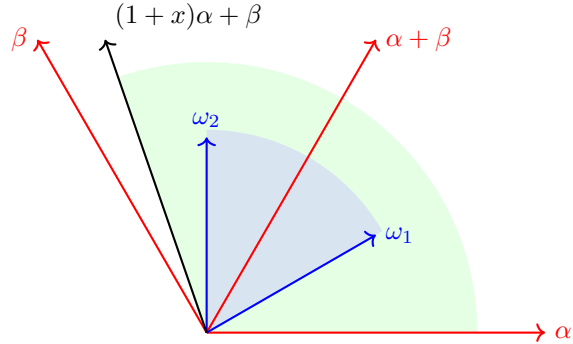
\end{example}

We can generalize the observations of the example in the following theorem.

\begin{thm}\label{thm:seminormal}
 The minimum  $A_{\min}$  attained in Problem \ref{conj:KLLeftcellTypeA} is unique and yields the base-change matrix to Young's seminormal basis.
\end{thm}

\begin{proof}
We first describe $A_{min}$ explicitly and then show that it yields the base-change matrix to Young's seminormal basis. The gradient of $f$ is given as $$\nabla f(A)=2GA.$$ We consider $\nabla f(A^*)=0,$ and deduce that the strictly upper-triangular part of $GA^*$ vanishes. Now, using the Cholesky decomposition for symmetric positive-definite matrices, we can find $L$ upper triangular such that $$G=L^tL.$$It follows that $LA^*$ is a diagonal matrix. Since $A^*$ is assumed to be upper unitriangular it is unique with this property and $$(A^*)^tGA^*=D,$$for some diagonal matrix $D$. Since $A^*$ is unique with this property and Young's seminormal basis can be characterised this way, see \cite{Murphy81}, we deduce that $A^*$ yields the base-change matrix to it. Moreover, we have $(A^*)^{-1}A_sA^*\geq 0,$ since the only non-negative entries in the seminormal form appear on the diagonal and are always greater than or equal to $-1$. It remains to show that the critical point $A^*$ indeed yields $A_{min}$ but this follows immediately from the positive-definite property of $G$ and the convexity of $f$.
\end{proof}

\begin{rem}
    The proof above yields a way to compute the base-change matrix to Young's seminormal basis by means of Gram-Schmidt. This suggests that instead of initialising with the polytabloid basis, we can also start with Young's seminormal form, see the proof of Theorem \ref{thm:reflection_rep_optimiser} on how to transform the problem. The Gram matrix in this case is diagonal and can be efficiently computed, see \cite[p.127]{JamesKerber1981}. Another positive side effect is that the matrices $1+s$ in Young's seminormal basis are sparser compared to the matrices in the polytabloid basis.
    Note that there are other conventions for Young's seminormal basis in the literature, which are also shown to be upper triangular but with non-trivial diagonal, see \cite{ArmonHalverson2021}. Moreover, the Gelfand--Tsetlin basis which agrees with Young's seminormal basis can be generalised and also shown to be upper triangular with regard to certain orders on the basis elements, see \cite{HaidarYacobi2025}.
\end{rem}

Computing the maximum in Problem \ref{conj:KLLeftcellTypeA} turns out to be more challenging. Since $f$ is convex, it follows that the maximum is achieved on the boundary of the feasible region. However, for specific cases when we know the Gram matrix of the Kazhdan--Lusztig basis and the representation matrices, we can show that it yields a KKT-point.

\begin{thm}\label{thm:reflection_rep_optimiser}
   For partitions $\lambda=(n,1)$, the base-change matrix to the Kazhdan--Lusztig representation yields a critical point of the optimization problem in Problem \ref{conj:KLLeftcellTypeA}.
\end{thm}

\begin{proof}
We first transform the optimization problem in Problem \ref{conj:KLLeftcellTypeA} into the following one:
Let $B$ denote the unitriangular base-change matrix from the polytabloid basis to the Kazhdan--Lusztig basis and consider$$\tilde{f}(A)=\mathrm{Trace}(A^tB^tGBA)$$subject to
$$C_s(A)=A^{-1}B^{-1}A_sBA\geq 0.$$ A maximum $A_{max}$ to  this problem can be transformed into a solution of the former problem by considering  $BA_{max}$. We want to show that the $d\times d$ identity matrix $A^*=I_d$ yields a critical point. Since we consider the maximum, it follows that it is achieved on the boundary and some of the constraints are active. This means we have $C_s(A)_{i,j}=0$ for some $i,j$. Let $\mathcal{C}$ denote the set of active constraints consisting of triples $(s,i,j)$ with $s\in S,1\leq i,j\leq d$.
We write $B^{-1}AB=\tilde{A}_s$ and $\tilde{G}=B^tGB$. The Karush-Kuhn-Tucker conditions, see \cite{BazaraaSheraliShetty2006},  state that a critical point $A$ of the problem satisfies the following identity
$$\nabla \tilde{f}(A)+\sum_{(s,i,j)\in \mathcal{C}} \mu_{s,i,j}\nabla C_s(A)_{i,j}=0,$$ with $\mu_{s,i,j}\geq 0$. We have$$\nabla \tilde{f}(A)_{u,w}=2(\tilde{G}A)_{u,w},$$for $1<u<w\leq d$, and hence $$\nabla \tilde{f}(A^*)=2\tilde{G}.$$For the constraints, we have$$\nabla C_s(A)_{u,w}=A^{-1}\tilde{A}_sE_{u,w}-A^{-1}E_{u,w}A^{-1}\tilde{A}_sA,$$where $E_{u,w}$ is the matrix with $1$ in position $(u,w)$ and zero everywhere else. We evaluate at $A=A^*$ and obtain $$\nabla C_s(A^*)=\tilde{A}_sE_{u,w}-E_{u,w}\tilde{A}_s=[\tilde{A}_s,E_{u,w}],$$where $[\cdot,\cdot]$ denotes the usual Lie bracket/commutator on matrices. For an active constraint $(s,i,j)\in \mathcal{C},$ we need to consider
$$(\nabla C_s(A^*))_{i,j}=[\tilde{A}_s,E_{u,w}]_{i,j}=(\tilde{A}_s)_{i,u}\delta_{w,j}-(\tilde{A}_s)_{w,j}\delta_{i,u}=[\tilde{A}_s^t,E_{i,j}]_{u,w},$$where $\delta_{i,j}$ denotes the Kronecker-delta. Now, assume we are in the case $\lambda=(n,1)$. Then $\tilde{G}$ is the Cartan matrix of type $A_n$, i.e. $\tilde{G}_{i,i+1}=-1$ for $1\leq i\leq n-1$. Using the action on the Kazhdan--Lusztig basis, see Example \ref{ex:simple_roots}, we have that$$[\tilde{A}_{s_{i+1}}^t,E_{i,i+1}]=2\cdot E_{i,i+1}.$$ Thus, we can conclude that $A^*$ yields a critical point by choosing $\mu_{s_{i+1},i,i+1}=1$ and setting the remaining coefficients to zero.
\end{proof}

\section{An Optimization Problem for Computing Canonical Bases of Irreducible \texorpdfstring{$\mathfrak{sl_n}$}{sl} Modules}\label{sec:canonical}

When considering the Kazhdan--Lusztig basis, it is striking to see the many similarities to canonical bases. In his groundbreaking paper \cite{Lusztig90}, Lusztig showed the existence of canonical bases for quantum groups in simply-laced types with remarkable non-negativity properties. More precisely, after setting $q=1$, one obtains bases for any finite-dimensional irreducible module such that the simple raising and lowering operators $E_i, F_i$ act non-negatively. Lusztig \cite{Lusztig91,Lusztig93} and Kashiwara \cite{Kashiwara91} extended the definition of the canonical bases using geometric and algebraic methods (respectively) to all symmetrizable Kac-Moody algebras. However, the non-negativity of the operator actions only persists in symmetric type.

Another type of canonical basis was discovered earlier by Gelfand and Tsetlin in type $A$, see \cite{GelfandTsetlin1950}. Here, one has explicit formulas for the operators $E_i, F_i$. Moreover, this basis is characterised by being orthogonal with respect to the Shapovalov form, the unique (up to scalar) contravariant symmetric form defined on any irreducible module, see also \cite{Molev2006}. Furthermore, the operators act non-negatively with respect to the Gelfand--Tsetlin basis. Consequently, both the Gelfand--Tsetlin basis and the canonical basis yield invariant cones under the action of the operators $E_i, F_i$. The base-change matrix from the Gelfand--Tsetlin basis to the canonical basis is triangular (with regard to some ordering) and coefficients only appear in the same weight space, see \cite{MolevYakimova2021}. Moreover, we can observe that the base-change matrix from the standard monomial basis to the canonical basis (similar to the $S_n$ case) is given by parabolic Kazhdan--Lusztig polynomials, see \cite{Du92,Du95}.

This leads us to the following optimization problem, analogous to the one presented in the previous section.

Fix $n \in \mathbb{N}$ and let $\lambda$ be a partition with fewer than $n$ parts. We consider the finite-dimensional irreducible $\mathfrak{sl}_n$-module $V(\lambda)$. Let $G$ be the Gram matrix of the Gelfand--Tsetlin basis of $V(\lambda)$.

\begin{optproblem}\label{prob:canonical}
    Consider the problem of maximizing or minimizing
    \[ f(B) = \operatorname{Tr}(B^t G B) \]
    subject to the constraints:
    \begin{enumerate}
        \item $B$ is lower unitriangular;
        \item The operators $E_i, F_i$ in the basis transformed by $B$ are non-negative.
    \end{enumerate}
    Let $B_{\max}$ and $B_{\min}$ be a maximizer and a minimizer of $f$ under these constraints, respectively.
\end{optproblem}

Similarly to Theorem \ref{thm:seminormal}, we can show that $B_{\min} = I$ and thus the Gelfand--Tsetlin basis yields the minimum in Problem \ref{prob:canonical}. Moreover, we can apply again Proposition \ref{prop:compact}, to show that the feasibility region is compact and a maximum exists.

\begin{problem}
    Which basis does $B_{\max}$ correspond to?
\end{problem}

\begin{rem}
    We have verified that $B_{\max}$ in Problem \ref{prob:canonical} is unique and yields the base-change matrix to the canonical basis (after suitable diagonal scaling) for all $\mathfrak{sl}_4$ modules $V(\lambda)$ indexed by partitions $\lambda \vdash m$ with $m \le 10$ and fewer than four parts. The resulting basis coincides (after the same diagonal rescaling) with the canonical basis obtained from the algorithm of Leclerc and Toffin \cite{LeclercToffin2000} for irreducible $\mathfrak{sl}_n$-modules. Beyond low rank, $B_{\max}$ need not coincide with the canonical basis and may instead detect other distinguished positive bases or geometric phenomena. Moreover, via the Shapovalov form one may identify $V(\lambda)$ with its dual, and it is therefore also natural to compare $B_{\max}$ with the dual canonical basis.
    Other potential candidates include the MV basis and dual semicanonical basis, which both differ from the dual canonical basis in general, see \cite{BKK2021}. It is known that the MV basis also yields an invariant cone. The non-negativity of the action of the operators $E_i$ is proven in \cite{BKK2021} and by other means one can show the non-negativity of the operators $F_i$.\footnote{This was communicated by Pierre Baumann and Joel Kamnitzer to the authors.} Whether the Chevalley operators act non-negatively in the dual semicanonical basis remains open. For the case $(4,4,2,2,0,0)$ for $\mathfrak{sl}_6$ these three bases differ and the base-change matrix from the MV to dual canonical and from dual canonical to dual semicanonical has only non-negative entries, respectively. In general it is conjectured that the base change from the MV basis to the dual canonical basis is given by non-negative integers, see \cite[Conjecture 2.15]{Leroux-Lapierre2025}.
    
    It might be interesting to test Problem \ref{prob:canonical} in other types. However, in other types (e.g., type $D$), we do not have an explicit construction for a basis which is simultaneously orthogonal and weight-space preserving, see \cite{Molev2006}.
\end{rem}

\newpage
\appendix
\section{Proof of Proposition \ref{prop:compact}} \label{appendix}
For convenience, we restate Proposition \ref{prop:compact} below.

\CompactProp*

The main difficulty in proving Proposition \ref{prop:compact} is to show boundedness. We use a  proof by contradiction, assuming that $\mathcal{F}$ is unbounded and constructing a non-trivial common invariant subspace. This involves a limit cone argument on subsets of the $(d-1)$-sphere  $\mathbb{S}^{d-1}\subseteq \mathbb{R}^d$. For this, we need to use the Hausdorff metric defined on subsets of $A,B\subseteq\mathbb{S}^{d-1}$ as $$d(A,B)=\max \{\sup_{a\in A}\inf_{b\in B} d(a,b),\sup_{b\in B}\inf_{a\in A} d(a,b)\},$$
where $d(a,b)=\lVert a-b\rVert$ denotes the usual metric induced by the standard Euclidean norm on $\mathbb{R}^d$.
We also use the following classical result on the topology of the space of compact subsets of a given compact set, e.g.\ the sphere $\mathbb{S}^{d-1}$, which can be, for instance, found in \cite[Theorem 3.2.4(3)]{Beer93}.  
\begin{thm}\label{thm:compact}
 The space $$\mathcal{K}(\mathbb{S}^{d-1})=\{X\subseteq \mathbb{S}^{d-1} \mid \emptyset\neq X \text{ is compact}\}$$ is compact with respect to the Hausdorff metric. In particular, any sequence of  non-empty compact subsets of  $\mathbb{S}^{d-1}$ has a converging subsequence.
\end{thm}

We are now equipped to prove Proposition \ref{prop:compact}.

\begin{proof}[Proof of Proposition \ref{prop:compact}]
    We first show that $\mathcal{F}$ is a closed semialgebraic set. First note that for a unitriangular matrix $U\in \mathbb{R}^{d\times d}$, the entries of $U^{-1}$ are polynomials in the entries of $U$ (this follows for example from Cramer's rule). Now, it follows that the inequalities $$U^{-1}A_kU\geq 0$$ yield polynomial inequalities that define a closed semialgebraic set.
It remains to show that $\mathcal{F}$ is bounded. Assume that $\mathcal{F}$ is unbounded, and consider a sequence $U_n\in \mathcal{F}$ with $\lVert U_n \rVert \to \infty$ as $n\to\infty$. Here, we take the operator norm, i.e.\ $$\lVert U_n \rVert = \max_{\lVert x\rVert=1}\lVert U_nx\rVert,$$where we use the standard Euclidean norm on the right. We derive a contradiction by  constructing a non-trivial common invariant subspace for the $A_k,k=1,\dots,r$.  Write $u_{n,1},\dots,u_{n,d}$ for the columns of $U_n$ and define $v_{n,j}=\frac{u_{n,j}}{\lVert u_{n,j} \rVert}$ for $j=1,\dots,d$, yielding the columns of a matrix $V_n$. After passing to a subsequence if necessary, we can assume $V_n \to V$ for some upper triangular $V$. Since at least one $\lVert u_{n,j} \rVert$  must diverge (otherwise the sequence $U_n$ is bounded), we get a zero entry on the diagonal of $V$. Hence, $V$ is singular. Notice that $v_{n,1}=e_1$ (column with one in the first entry and zero everywhere else). It follows that $$0\subsetneq \im V=V\mathbb{R}^d\subsetneq \mathbb{R}^d.$$In the remaining part of the proof, we show that $\im V$ spans a non-trivial common invariant subspace for $A_k,k=1,\dots,r$, giving the contradiction.
We consider $$C_n=U_n\mathbb{R}_{\geq0}^d=V_n\mathbb{R}_{\geq0}^d,$$ which gives a sequence of proper invariant cones. In order to construct a limit cone $C$, we first consider the compact sets $K_n=C_n\cap \mathbb{S}^{d-1}$. By Theorem \ref{thm:compact}, we can assume that $K_n \to K$ for some non-empty compact $K\subseteq \mathbb{S}^{d-1}$. We define $C$ as the convex cone generated by $K$, i.e.\ $$C=\mathrm{cone}(K)= \Big\{ \sum_{i=1}^\ell a_i k_i \Bigm| a_i \in \mathbb{R}_{\geq0}, k_i \in K \Big\}.$$Note that $v_{n,j}\in K_n$ for all $j,n$ and since $K$ is closed it follows that $v_j \in K$ for all $j$. This implies that $$\im V\subseteq \mathrm{span}_\mathbb{R}(K)=\mathrm{span}_\mathbb{R}(C).$$ The idea is to first show that $C$ is an invariant cone and pointed, i.e.\ $C\cap -C=\{0\}$, and then show that $\im V=\mathrm{span}_\mathbb{R}(C)$. Invariance of $C$ then implies invariance of $\im V=\mathrm{span}_\mathbb{R}(C)$.

$C$ is invariant: We first show that for all $x\in K$ and all $k=1,\dots,r$, we have $A_kx\in C$. Let $x\in K$ and choose $x_n\in K_n$ with $x_n\to x$ as $n\to\infty$. Since each $C_n$ is invariant under $A_k$, we have $A_kx_n\in C_n$ for all $n$. If $A_kx=0$, there is nothing to show. Otherwise, since $A_kx_n\to A_kx\neq 0$, we have $A_kx_n\neq 0$ for all sufficiently large $n$. Hence $$\frac{A_kx_n}{\lVert A_kx_n\rVert}\in K_n$$ for all sufficiently large $n$. Passing to the limit and using that $K_n\to K$ in the Hausdorff metric, we obtain $$\frac{A_kx}{\lVert A_kx\rVert}\in K.$$ Thus $A_kx\in C$. Now let $x\in C$. Since $C=\operatorname{cone}(K)$, we can write $x=\sum_{i=1}^\ell a_i k_i$ with $a_i\in\mathbb R_{\ge 0},\ k_i\in K.$ By the first part, $A_kk_i\in C$ for all $i$. As $C$ is a convex cone, it follows that $$A_kx=\sum_{i=1}^\ell a_i A_kk_i\in C.$$ Hence $C$ is invariant.

$C$ is pointed: First note that $C\cap -C$ yields a common invariant subspace for the $A_k$. It follows that either $C\cap -C=\{0\}$ or $C\cap -C=\mathbb{R}^d$. Since the last row of each $U_n$ equals $e_d^t,$ it follows that $C_n \subseteq \{ x \in \mathbb{R}^d \mid x_d\geq 0 \}=\mathbb{H}_d$. Likewise, we can conclude that $K_n \subseteq \mathbb{H}_d$ and since this property is preserved under Hausdorff limits, it follows that $K \subseteq \mathbb{H}_d$ and thus $C \subseteq \mathbb{H}_d$. Hence, we have $C\cap -C\subseteq \{x\in \mathbb{R}^d\mid x_d=0\}\subsetneq \mathbb{R}^d$, which implies that $C\cap -C=\{0\}$. Thus, $C$ is pointed.

$\im V=\mathrm{span}_\mathbb{R}(C)$: We have already seen that $\im V\subseteq \mathrm{span}_\mathbb{R}(C).$ Thus, it suffices to show that  $\mathrm{span}_\mathbb{R}(C)\subseteq \im V.$ Note that for this it suffices to show that $K\subseteq \im V$. Let $x\in K$ and $x_n \in K_n$ with $x_n \to x$. There exist $\lambda_{n,j}\geq 0$ with $$x_n = \sum_j \lambda_{n,j}v_{n,j}.$$In order to pass to the limit $n\to \infty$ and get a similar expression for $x$, we need to show that for each $j$, $\lambda_{n,j}$ is bounded. Since $\lambda_{n,j}\geq 0$, this is equivalent to showing that the sequence defined by $$M_n=\sum_j \lambda_{n,j}$$is bounded. Assume $M_n$ is unbounded. We consider $$\mu_{n,j}=\frac{\lambda_{n,j}}{M_n}\geq0   \text{ with }\mu_{n,j}\geq0,\sum_j \mu_{n,j}=1.$$  Since the $\mu_{n,j}$ yield a bounded sequence for each $j$, we can pass to a convergent subsequence such that$$\mu_{n,j}\to \mu_j  \text{ with }\mu_j\geq0 ,\sum_j \mu_j=1.$$ It follows that $$\frac{x_n}{M_n}=\sum_j\mu_{n,j}v_{n,j}\overset{n\to \infty}{\longrightarrow} 0=\sum_j \mu_jv_j.$$Since $\mu_j\geq0 ,\sum_j \mu_j=1$, and $\lVert v_j \rVert=1$ for all $j$, there exists $i$ with $\mu_i>0$ such that $$-\mu_i v_i = \sum_{j\neq i}\mu_jv_j \in C.$$Hence, $C\cap -C \neq \{0\}$ contradicting pointedness of $C$. It follows that $M_n$ is bounded. After passing to a subsequence, we  can take $$\lambda_{n,j}\to \lambda_j \text{ with } x=\sum_j\lambda_j v_j.$$

Now, we can conclude the proof. Since $K\subseteq \im V,$ we have $C\subseteq \im (V)$. Thus, from the previous steps, we conclude that $\im V=\mathrm{span}_\mathbb{R}(C)$ yields a proper invariant subspace, contradicting the irreducibility of the matrix family $A_1,\dots,A_r$. It follows that $\mathcal{F}$ is bounded and hence a compact semialgebraic set.
\end{proof}

\section*{Acknowledgements}

This project was supported by the Australian Government through the Australian Research Council's Discovery Projects funding scheme (project DP230102982).

\newpage

\printbibliography

\end{document}